\documentclass[a4paper,12pt]{article}
\usepackage{amsmath,amsthm,amsfonts,amssymb,bbm}
\usepackage{graphicx,psfrag,subfigure,color,cite}
\usepackage[UKenglish]{babel}
\usepackage{textcmds}
\usepackage{mathtools}
\usepackage{float}
\usepackage{changes}
\usepackage[colorlinks]{hyperref}

\numberwithin{equation}{section}

\usepackage[margin=1.2in]{geometry}

\newcommand{\eps}{\varepsilon}

\newtheorem{prop}{Proposition}[section]
\newtheorem{assumption}[prop]{Assumption}
\newtheorem{thm}[prop]{Theorem}
\newtheorem{lem}[prop]{Lemma}

\newtheorem{cor}[prop]{Corollary}

\newtheorem{cla}[prop]{Claim}

\newtheorem{rem}[prop]{Remark}
\newenvironment{remark}{\begin{rem}\normalfont}{\end{rem}}
\newtheorem{example}[prop]{Example}

\makeatletter
\newcommand{\hathat}[1]{%
	\begingroup%
	\let\macc@kerna\z@%
	\let\macc@kernb\z@%
	\let\macc@nucleus\@empty%
	\hat{\mathchoice%
		{\raisebox{.3ex}{\vphantom{\ensuremath{\displaystyle #1}}}}%
		{\raisebox{.3ex}{\vphantom{\ensuremath{\textstyle #1}}}}%
		{\raisebox{.16ex}{\vphantom{\ensuremath{\scriptstyle #1}}}}%
		{\raisebox{.14ex}{\vphantom{\ensuremath{\scriptscriptstyle #1}}}}%
		\smash{\hat{#1}}}%
	\endgroup%
}
\makeatother
\title{TASEP with a general initial condition and a deterministically moving wall}
\author{Sabrina Gernholt\thanks{Institute for Applied Mathematics, Bonn University, Endenicher Allee 60, 53115 Bonn, Germany. Email: {\tt sgernhol@uni-bonn.de}}
}
\date{}

\begin{document}
\maketitle

\begin{abstract}
We study the totally asymmetric simple exclusion process (TASEP) on $\mathbb{Z}$ with a general initial condition and a deterministically moving wall in front of the particles. Using colour-position symmetry, we express the one-point distributions in terms of particle positions in a TASEP with step initial condition along a space-like path. 

Based on this formula, we analyse the large-time asymptotics of the model under various scenarios. For initial conditions other than the step initial condition, we identify a distinct asymptotic behaviour at the boundary of the region influenced by the wall, differing from the observations made in \cite{BBF21,FG24}.
Furthermore, we demonstrate that product limit distributions are associated with shocks in the macroscopic empirical density.

As a special case of our starting formula, we derive a variational expression for the one-point distributions of TASEP with arbitrary initial data. Focusing on non-random initial conditions, such as periodic ones with an arbitrary density, we leverage our analytical tools to characterise the limit distribution within the framework of particle positions.
\end{abstract}

\section{Introduction} 

The totally asymmetric simple exclusion process (TASEP) was first introduced by Spitzer in 1970 \cite{Spi70} and is one of the most extensively studied interacting particle systems. We consider a single-species, continuous-time TASEP on the integer lattice $\mathbb{Z}$. This model consists of particles positioned on the lattice, each attempting to jump one step to the right after an exponential waiting time with rate one. The waiting times for each particle are independent of one another. Each site can host at most one particle at a time, and a jump is successful only if the target site is empty. 

We are interested in the large-time fluctuations of particle positions in a TASEP with a general initial condition and a rightmost particle that is blocked by a moving wall. Specifically, the wall moves to the right deterministically in front of the system of particles, preventing any jumps across it. For the particular case of the step initial condition, where all sites $...,-3,-2,-1,0$ are occupied by particles and all sites $1,2,3, \dots$ are empty, this setting has been studied by Borodin-Bufetov-Ferrari in \cite{BBF21} and by Ferrari-Gernholt in \cite{FG24}. 

With its interpretation as a random growth model, TASEP is one of the best-understood representatives of the KPZ universality class \cite{KPZ86} in $1+1$ dimensions. It exhibits non-trivial asymptotic fluctuations in the $1:2:3$ scaling limit, meaning that for large times $t$, the spatial correlations are of order $t^{2/3}$ and the fluctuations are of order $t^{1/3}$. 
For a variety of initial conditions, its spatial scaling limit has been identified and defined as the KPZ fixed point \cite{MQR17}. When including the temporal component, the overarching universal scaling limit is the directed landscape \cite{DOV22,DV21}. 

By the convergence to the KPZ fixed point, Matetski-Quastel-Remenik \cite{MQR17} provide a variational expression for the one-point limit distributions of TASEP. See also \cite{CLW16} for the one-point distributions of the discrete-time parallel update TASEP and \cite{BL13,CFS16,FO17,Jo03b,QR13b,QR13,QR16} for other instances of such formulas. 
For TASEP with step initial condition and a deterministic wall, the one-point limit distribution has the same form. In Theorem~4.4 of \cite{BBF21}, it is expressed in terms of a variational process, under the assumption that the wall influences the fluctuations of the tagged particle in only one macroscopic time region. Theorem~2.5 of \cite{FG24} extends this result to multiple time regions, where the asymptotic distribution becomes a product of variational formulas. 

\paragraph{New results on TASEP with a wall.}
The main goal of this work is to derive the large-time one-point distributions of TASEP with a wall under generic initial conditions, in terms of variational expressions. 
To achieve this, we establish a novel finite-time identity that expresses the one-point distributions of the process in terms of particle positions in a TASEP with step initial condition along a space-like path. As detailed in Theorem~\ref{th:finite_time_identity_non_random_IC}, we obtain
\begin{equation} \label{eq:introduction_finite_time}
\begin{aligned} 
\mathbb{P} (x^f_n(T) > s) = \mathbb{P} ( &x_n(t) > s - f(T-t) \ \forall t \in [0,T], \\ & x_{n-j}(T) > s - x_{1+j}^f(0) \ \forall j \in \{0, \dots, n-1\}),
\end{aligned} 
\end{equation}
where $x^f_n(T)$ denotes the position of the $n$-th particle in a TASEP with wall $f$ at time $T$, and $(x(t),t\geq 0)$ represents the particle positions in a TASEP with step initial condition, without a wall restriction. The formula~\eqref{eq:introduction_finite_time} is valid for any deterministic or random initial data and serves as the foundation for a large-time analysis. 

In Theorem~\ref{th:2.1}, we observe that particles close enough to the wall exhibit the same asymptotic fluctuations as in the step initial condition case studied in \cite{BBF21,FG24}. In contrast, for particles far away from the wall, the asymptotic fluctuations depend on the specific initial condition. More precisely, for the leftmost particles affected by the wall, their law either takes a product form or interpolates between the scaling limit without wall and the one under wall influence. We refer to these two situations as \qq{decoupling} and \qq{interpolation}. The corresponding results are Theorem~\ref{th:2.2} and Theorem~\ref{th:2.3}. In both cases, the asymptotic laws are expressed by variational formulas such as 
\begin{equation} \label{eq:introduction_var_1}
\mathbb{P} \left ( \sup_{\tau \in \mathbb{R}} \{ \mathcal{A}_2(\tau)-g_1(\tau)\}<c_1 s\right) \times \mathbb{P} \left ( \sup_{\tau \in \mathbb{R}} \{ \mathcal{A}_2(\tau)-g_2(\tau)\}<c_2 s\right)
\end{equation}
or 
\begin{equation} \label{eq:introduction_var_2}
\mathbb{P} \left ( \sup_{\tau \geq 0} \{ \mathcal{A}_2(\tau)-g_1(\tau)\}<c_1 s, \ \sup_{\tau \leq 0} \{ \mathcal{A}_2(\tau)-g_2(\tau)\}<c_2 s\right),
\end{equation}
where $\mathcal{A}_2$ denotes an Airy$_2$ process, $c_1,c_2 > 0$ are constants, and $g_1, g_2$ are piecewise continuous càdlàg functions with at least quadratic growth. One of these functions is determined by the wall constraint, while the other is shaped by the initial condition.

We focus on deterministic initial data, and our results apply to all non-random initial conditions that locally, within a $(t^{1/3} \times t^{2/3})$-window, fluctuate around a half-$d$-periodic initial condition with $d \in \mathbb{R}_{\geq 1}$. Theorem~\ref{th:2.1} and Theorem~\ref{th:2.2} require even fewer assumptions. Furthermore, Section~\ref{sect:main_results_asymptotics_random} outlines how our results transfer to random initial conditions. 

Our analysis requires the weak convergence of the particle process with varying labels at a fixed time in a TASEP with step initial condition to the Airy$_2$ process, on the fluctuation scale. In the framework of particle positions, this result was not available, and we establish it in Lemma~\ref{lem:weak_conv}.

In Section~\ref{sect:main_shocks}, we show that the product limit distributions appearing in the decoupling case~\eqref{eq:introduction_var_1} in this paper and in \cite{FG24} for multiple time regions are related to shocks, that is, discontinuities in the macroscopic density of the process. This connection is well-known for shocks generated purely by initial conditions in related models without a wall, as discussed in \cite{BF22, Fer18, FN13, FN16, FN19, FN24, N17, Nej21, QR18}. For TASEP with a wall, it has not yet been explored.

\paragraph{Key techniques and related work.}
The identity~\eqref{eq:introduction_finite_time} from Theorem~\ref{th:finite_time_identity_non_random_IC} is the starting point of our analysis. Previously, the formula was established in the specific case of the step initial condition in Proposition~3.1 of \cite{BBF21}, where it reduces to an expression in terms of a tagged particle process. 
Both formulas are derived by coupling the TASEP with a wall to a multi-species TASEP and applying the colour-position symmetry. This symmetry originates in \cite{AHR09} and has been examined in \cite{AAV11,BB19,BW18,Buf20,Gal20}. In comparison to the step initial condition case, the proof of Theorem~\ref{th:finite_time_identity_non_random_IC} is much more involved because the general initial data must be taken into account.

Based on Theorem~\ref{th:finite_time_identity_non_random_IC}, we analyse the large-time fluctuations of the tagged particle under wall influence in the three aforementioned situations by employing several properties of the TASEP with step initial condition. We now discuss the ingredients required when considering the leftmost particles affected by the wall.

In the decoupling case, we recall the concept of backwards paths developed in \cite{Fer18} and \cite{BBF21,FG24,FN19}. As elaborated in \cite{FG24}, combined with a slow decorrelation result \cite{Fer08}, the localisation of these paths is a powerful tool to achieve the asymptotic independence of particle fluctuations in TASEP. We refer to \cite{FG24} for the details, as our arguments are based on those presented there.  

In the interpolation case, we rely on the convergence of the rescaled particle positions in a TASEP with step initial condition along a space-like path to the Airy$_2$ process. For finite-dimensional distributions, this convergence has been shown in the more general context of PushASEP in \cite{BF07}. See also \cite{BFS07b} for corresponding results for the PNG model and the discrete TASEP with parallel update. In \cite{CFP10c,CFP10b,Fer08}, it is explained how slow decorrelation can be used to generalise convergence results to all space-time paths except for characteristic lines. Further, \cite{DOV22} observes that general convergence along space-like paths is also a consequence of the convergence to the directed landscape. 

We work within the framework of particle positions, as our finite-time identity is naturally expressed in this setting, and we already have the results of \cite{BBF21,FG24} at hand. We combine the results of \cite{BF07} with weak convergence results for (a) the tagged particle process and (b) the process of particles with varying labels at a fixed time in a TASEP with step initial condition. The tagged particle process has also been addressed in \cite{SI07}, and weak convergence is established in Proposition~2.9 of \cite{BBF21}. In the case of the step initial condition, the identity~\eqref{eq:introduction_finite_time} reduces to an expression involving only the tagged particle process. For this reason, \cite{BBF21,FG24} analysed (a) but not (b). To fill this gap, in Lemma~\ref{lem:weak_conv}, we confirm weak convergence for the fixed time process as well. Similar results were previously proven for related models in \cite{FO17,Jo03b}.  

The key to weak convergence is the tightness of the processes. For both (a) and (b), this is obtained by a comparison to a stationary TASEP in the manner of \cite{CP15b}, see also \cite{BF22,FO17}. The comparison inequalities are given in \cite{FG24}. In view of our variational starting formula, they are essential for narrowing down the subset of events that contribute non-trivially to the limiting probability. It is therefore natural to invoke the comparison inequalities in the proof of Lemma~\ref{lem:weak_conv} as well, rather than deriving it from known weak convergence results through slow decorrelation. 

\paragraph{Related results without a wall.}
A noteworthy by-product of our work is a variational formula for the limiting one-point distributions of a TASEP with one of the above initial conditions also in the absence of a wall. 

As a direct consequence of the identity~\eqref{eq:introduction_finite_time}, we obtain a variational expression for the one-point distributions in terms of a TASEP with step initial condition. We also provide a short self-contained proof utilising the coupling introduced in \cite{Sep98c} and colour-position symmetry. Regarding related models, these arguments play a similar role as the symmetry in LPP models, see \cite{CLW16,FO17,Jo03b}, and the skew-time reversibility of TASEP height functions, see \cite{MQR17}.

The key components for the large-time asymptotics are the weak convergence result in Lemma~\ref{lem:weak_conv} and again the comparison to a stationary TASEP. We present the limiting result in Proposition~\ref{pro:limit_thm_2.2} and extend it to initial conditions with particles distributed across the entire lattice in Corollary~\ref{cor:prop_2.2}. These results apply, for example, to the TASEP with (half-)$d$-periodic initial condition, where $d \in \mathbb{R}_{\geq 1}$. In this regard, previous findings include \cite{BFP06} with $d \in \mathbb{N}$ for the discrete TASEP, \cite{BFPS06,BFS07} for the TASEP with (half)-flat initial condition, and \cite{FO17} for the LPP setting and general densities. In Section~\ref{sect:main_results_asymptotics_random}, we explain how Proposition~\ref{pro:limit_thm_2.2} and Corollary~\ref{cor:prop_2.2} can be extended to random initial data such as Bernoulli initial conditions, providing a basis to confirm the asymptotic results of \cite{BFP09,CFS16,CLW16}.

\paragraph{Outline.} We present our main results in Section~\ref{sect:main_results}. As a special case, Section~\ref{sect:half_d_periodic_TASEP} discusses the large-time behaviour of a TASEP with half-$d$-periodic initial condition and a wall. The remainder of our work focuses on the proofs of the results from Section~\ref{sect:main_results}. In Section~\ref{sect:proof_finite_time_identity}, we establish the variational starting formula, while Section~\ref{sect:proofs_large_time} contains the proofs of our large-time results. Section~\ref{sect:proof_density_shock} explores the relationship between product limit distributions and shocks in the macroscopic empirical density. Finally, in Section~\ref{sect:auxiliary results}, we provide the weak convergence result and apply the comparison to a stationary TASEP. Appendix~\ref{appendix} includes technical computations and known one-point estimates for the TASEP with step initial condition. 

\paragraph{Acknowledgements.} I would like to thank P.L. Ferrari for insightful discussions and valuable feedback. I am also grateful to A. Borodin for his helpful comments on an earlier version of this work. Finally, I thank the two anonymous referees for their careful reading and constructive suggestions. This work was partially supported by the Deutsche Forschungsgemeinschaft (DFG, German Research Foundation) through the Collaborative Research Centre \qq{The Mathematics of Emergent Effects} (CRC 1060, project-ID 211504053) and the Bonn International Graduate School of Mathematics at the Hausdorff Center for Mathematics (EXC 2047/1, project-ID 390685813).

\section{Main results} \label{sect:main_results} 

\subsection{Finite-time identity for TASEP with general initial condition and wall} \label{sect:main_results_finite_time_identity} 

Our first result is a finite-time identity for the one-point distributions of a TASEP with a general initial condition and a deterministic wall constraint. 

For this purpose, let $f:\mathbb{R}_{\geq 0} \to \mathbb{R}_{\geq 0}$ be a non-decreasing càdlàg function with $f(0) = 0$. In the following, all TASEPs are described via their associated particle position processes. We denote by $(x^f(t),t\geq 0)$ the particle positions in a TASEP with a (random or deterministic) initial condition \begin{equation}
x^f_n(0) = x_n^f, \ n \in \mathbb{N}
\end{equation}
and the wall constraint 
\begin{equation}
x^f_1(t) \leq f(t) \text{ for all } t \geq 0.
\end{equation}
As usual, $x_n^f(t)$ is the position of the $n$-th particle in the process at time $t$. As a convention, the labels of particles are increasing to the left, so that $ \cdots< x_{n+1}^f(t) < x_n^f(t) < \cdots$. In this setting, the initial condition contains a rightmost particle at the position $x_1^f$ with $x_1^f \leq 0$. 

Further, we let $(x(t),t\geq 0)$ be a TASEP with step initial condition and without wall restrictions, which is independent of $\{x_n^f, n \in \mathbb{N}\}$. Below, we will consider several TASEPs with step initial condition. Unless stated otherwise, it is always centred at the origin, meaning that $x_n(0)=-(n-1), n \in \mathbb{N}$. Our observation time variable is $T \geq 0$. 

\begin{thm} \label{th:finite_time_identity_non_random_IC} 
	For any $n \in \mathbb{N}$, it holds
	\begin{equation} \label{eq:thm_finite_time_identity_non_random_IC} \begin{aligned} 
	&\mathbb{P} (x^f_n(T) > s) \\  &= \mathbb{P} ( x_n(t) > s - f(T-t) \ \forall t \in [0,T], \ x_{n-j}(T) > s - x_{1+j}^f \ \forall j \in [n-1]),
	\end{aligned} \end{equation}
	where $[n-1] = \{0, \dots, n-1\}$.
\end{thm}

Theorem~\ref{th:finite_time_identity_non_random_IC}, proven in Section~\ref{sect:proof_finite_time_identity}, generalises the identity for step initial condition obtained in Proposition~3.1 of \cite{BBF21}. Like them, we couple the process with wall constraint to a multi-species TASEP and apply the colour-position symmetry. However, for a general initial condition, the multi-species TASEP now starts from some permutation other than the identity. The proof is more involved since we need to take care of the inverse permutation that is applied to a second multi-species TASEP after its evolution. 
The finite-time identity \eqref{eq:thm_finite_time_identity_non_random_IC} holds for each integer $s$. For the large-time analysis, we allow $s \in \mathbb{R}$, as constant shifts become negligible. 

\begin{remark} \label{re:1} For a TASEP $(\tilde{x}(t),t\geq 0)$ with the same initial condition, $\tilde{x}_n(0) = x_n^f$ for $n \in \mathbb{N}$, and without a wall constraint, Theorem~\ref{th:finite_time_identity_non_random_IC}\footnote{In the proof of Theorem~\ref{th:finite_time_identity_non_random_IC}, we can choose $f \equiv + \infty$.} implies that its one-point distributions are related to a TASEP with step initial condition: 
	\begin{equation} \label{eq:remark_1} 
	\mathbb{P}(\tilde{x}_n(T) > s) = \mathbb{P} \left ( \min_{j \in [n-1]} \{ x_{n-j}(T)+x_{1+j}^f \} > s \right).
	\end{equation}
	Using known convergence results for $(x(t), t \geq 0)$, the asymptotic one-point distributions can be expressed as variational formulas of the Airy$_2$ process, see Proposition~\ref{pro:limit_thm_2.2} and Corollary~\ref{cor:prop_2.2}. This eliminates the need to represent the one-point distribution as a Fredholm determinant and to perform asymptotic analysis on its kernel, as done in earlier works like \cite{BFS07} on the half-flat initial condition. While their arguments apply to joint distributions and densities $\tfrac{1}{d}$ with $d \in \mathbb{N}$, our formula restricts to one-point distributions, but works for $d \in \mathbb{R}_{\geq 1}$. Still, also this approach involves the use of technical tools, such as a comparison to a stationary TASEP \cite{FG24}. 
	
	With Lemma~\ref{lem:colour_position_symmetry_step}, we provide a self-contained proof of \eqref{eq:remark_1}, as the technical details from the proof of Theorem~\ref{th:finite_time_identity_non_random_IC} are not required here. In particular, we do not need the existence of a rightmost particle. For a general initial condition $\{x_n^f, n \in \mathbb{Z}\}$ instead of $\{x_n^f, n \in \mathbb{N} \}$, the minimum in \eqref{eq:remark_1} is taken over $j \leq n-1$. The key idea for \eqref{eq:remark_1} is that the colour-position symmetry allows to replace the countably many TASEPs with shifted step initial conditions in the well-known coupling of \cite{Sep98c}, see \eqref{eq:lemma_2.1_sep98}, by a single TASEP with step initial condition. 
\end{remark}  

\subsection{Large-time asymptotics for TASEP with non-random initial condition and wall} \label{sect:main_results_asymptotics} 

In this section, we apply Theorem~\ref{th:finite_time_identity_non_random_IC} to determine the asymptotic fluctuations in a TASEP with a general non-random initial condition and a deterministic wall constraint.

We consider a TASEP $(x^f(t),t\geq 0)$ with a non-random initial condition $\{x_n^f, \ n \in \mathbb{N} \}$ and a wall constraint as in Section~\ref{sect:main_results_finite_time_identity}. The TASEP with the same initial condition but without the wall constraint is denoted by $(\tilde{x}(t),t\geq 0)$. As before, $(x(t),t \geq 0)$ is a TASEP with step initial condition. From now on, we suppose that the decay of $x_n^f$ in $n$ is bounded linearly, meaning 
\begin{equation} x_n^f \geq - d n  -\text{const.} \label{eq:linear_decay_IC} \end{equation} 
for some $d \in \mathbb{R}_{\geq 1}$. The half-$d$-periodic initial condition $x_n^f = -\lfloor d (n-1) \rfloor$ is discussed in Section~\ref{sect:half_d_periodic_TASEP} as a special case.

In \eqref{eq:thm_finite_time_identity_non_random_IC}, we choose $n=\alpha T$ with $\alpha \in (0,1)$ and set $s = \xi T - S T^{1/3}$ with $S \in \mathbb{R}$ and $\xi$ such that $x^f_{\alpha T}(T) \simeq \xi T$. The interesting case is when the $\alpha T$-th particle experiences a significant wall influence until time $T$. Naturally, the parameter $\xi$ depends on the function $f$ and on the initial condition of the process.

We denote the macroscopic scaling of the $\alpha T$-th particle in the original process by $g_\alpha T$, meaning $\tilde{x}_{\alpha T}(T) \simeq g_\alpha T$.

We wish to highlight that \eqref{eq:linear_decay_IC} is a natural restriction, as we aim to observe a wall influence.
If $x_n^f$ decays faster than linearly in $n$, then the macroscopic position of $x^f_{\alpha T}(T)$ is bounded by $x^f_{\alpha T}+T \ll -\alpha T$ for large times. Then, $x_{\alpha T}(t) + f(T-t)$ is strictly larger than the scaling of $x^f_{\alpha T}(T)$ for all times $t \in [0,T]$. Consequently, taking the formula~\eqref{eq:thm_finite_time_identity_non_random_IC} into account, there is no observable wall influence.

The constant in \eqref{eq:linear_decay_IC} could also be replaced by a term in $o(T)$. 

Three different situations arise; see Figure~\ref{fi:example_limit_distr_periodic} for an illustration.

\paragraph{Situation 1: $\xi < g_\alpha$.}
If the scaling of the tagged particle affected by the wall is macroscopically smaller than its scaling in case of no wall influence, the events at time $T$ in \eqref{eq:thm_finite_time_identity_non_random_IC} hold with probability converging to $1$.

\begin{thm} \label{th:2.1} 
	Suppose $\xi < g_\alpha$. Then, it holds
	\begin{equation} \label{eq:situation_1} \begin{aligned} 
	& \lim_{T \to \infty} \mathbb{P}(x^f_{\alpha T}(T) > \xi T - S T^{1/3}) \\
	& = \lim_{T \to \infty} \mathbb{P}(x_{\alpha T}(t) > \xi T - f(T-t)-ST^{1/3} \ \forall t \in [0,T]).
	\end{aligned}
	\end{equation}
\end{thm}

Theorem~\ref{th:2.1}, proven in Section~\ref{sect:proof_asymptotics_interior}, holds true for any initial condition for which a weak law of large numbers is known for $\tilde{x}_{\alpha T}(T)$. Together with Proposition~3.1 of \cite{BBF21}, it implies that for any initial condition, in the interior of the region affected by the wall\footnote{Here, the \qq{interior} refers to the subinterval of $\alpha \in (0,1)$ such that the wall affects $x^f_{\alpha T}(T)$. Indeed, there is at most one $\alpha$ such that $x^f_{\alpha T}(T)$ is influenced by the wall and keeps its original scaling.} and for $\xi \in (-\alpha,1-2\sqrt{\alpha})$\footnote{If $\xi < - \alpha$, we cannot observe a wall influence. If $\xi = - \alpha$, we only obtain one if $f$ stays at the origin at least until time $(1-\alpha)T$. In a TASEP with step initial condition and with the same wall, the $\alpha T$-th particle does not move macroscopically in this case. We exclude this special case in the following, except for a comment on $f \equiv 0$ in Section~\ref{sect:half_d_periodic_TASEP}.}, we find the same convergence results like for a TASEP with step initial condition and with the same wall constraint. In this regard, we make the following assumptions: 

\begin{assumption} \label{assumpt:2.1.1} 
	Let $f$ be a non-decreasing càdlàg function on $\mathbb{R}_{\geq 0}$ with $f(0)=0$. Let $n \in \mathbb{N}_0$ and $\alpha_0 < \dots < \alpha_n \in (\alpha,1]$. We require: 
	\begin{itemize}
		\item[(a)] For some fixed $\eps \in (0,\alpha_0-\alpha)$ and for all $t \in [0,T]$ with $|t-\alpha_i T | > \eps T$ for $i = 0, \dots, n$, it holds 
		\begin{equation}
		f(T-t)\geq f_0((T-t)/T)T+K(\eps)T,
		\end{equation}
		where $K(\eps)$ is a positive constant independent of $t$ and the function $f_0:[0,1] \to \mathbb{R}$ is defined by 
		\begin{equation} \label{eq:fct_f0} 
		f_0(\beta) \coloneqq \begin{cases} \xi - \sqrt{1-\beta}(\sqrt{1-\beta}-2\sqrt{\alpha}), \ & \beta \in [0,1-\alpha), \\ 
		\xi+\alpha, \ & \beta \in [1-\alpha,1]. \end{cases} 
		\end{equation}
		\item[(b)] For $i \in \{0, \dots, n \}$, we define \[c_1^i  = \alpha^{-1/6} \alpha_i^{1/6} (\sqrt{\alpha_i}-\sqrt{\alpha})^{2/3}, \ c_2^i = 2 \alpha^{-1/3} \alpha_i^{5/6} (\sqrt{\alpha_i}-\sqrt{\alpha})^{1/3}\] \[ \text{and } \mu^i(\tau,T) = \sqrt{\alpha_i}(\sqrt{\alpha_i}-2\sqrt{\alpha})T-2\tau\alpha^{-1/3} \alpha_i^{1/3}(\sqrt{\alpha_i}-\sqrt{\alpha})^{4/3} T^{2/3}.\] We parametrize $T-t=(1-\alpha_i)T+c_2^i \tau T^{2/3}$ and let 
		\begin{equation} \label{eq:assumpt_2.1} 
		f(T-t) = \xi T - \mu^i(\tau,T) - c_1^i(\tau^2-g_T^i(\tau))T^{1/3}, \ \tau \in \mathbb{R}.
		\end{equation}
		The sequences $(g_T^i), i = 0,\dots,n $ converge uniformly on compact sets to piecewise continuous and càdlàg functions $g_i$. Further, there exists a constant $M \in \mathbb{R}$ such that for $T$ large enough, $g_T^i(\tau) \geq - M + \tfrac{1}{2} \tau^2$ for $|\tau| \leq \eps (c_2^i)^{-1} T^{1/3}$. 
	\end{itemize}
	If $\alpha_n = 1$, then we have $\xi = 1-2\sqrt{\alpha}$ and for $i=n$, we require Assumption (b) only for $\tau \geq 0$. 
\end{assumption}

As noted in \cite{BBF21,FG24}, Assumption~\ref{assumpt:2.1.1}(a) ensures that wall effects on the tagged particle arise only in the vicinity of the times $\alpha_i T$, for $i=0, \dots, n$. Consider the right hand side of equation \eqref{eq:situation_1}. For all times $t$ satisfying $|t-\alpha_i T| > \eps T$ for each $i=0, \dots, n$, the inequality $x_{\alpha T}(t) > \xi T - f(T-t)-ST^{1/3}$ holds with probability tending to $1$ as $T \to \infty$. In contrast, the sequences $(g_T^i)$ in Assumption~\ref{assumpt:2.1.1}(b) describe the $\mathcal{O}(T^{1/3})$-fluctuations of $\xi T - f(T-t)$ around the law of large numbers of $x_{\alpha T}(t)$ near the times $\alpha_i T$, $i=0, \dots, n$. Together with the parameters $\alpha$ and $\alpha_0, \dots, \alpha_n$, they determine the limit distribution. 

In Remark~2.6 of \cite{FG24}, it was pointed out that we can also allow limit functions $g_i$ that equal $+\infty$ in an interval. Given Theorem~\ref{th:2.1}, Theorem~4.4 of \cite{BBF21} as well as Theorem~2.5 and Remark~2.6 of \cite{FG24} imply 

\begin{cor} \label{cor:thm_2.1} 
	For $x^f_{\alpha T}(T) \simeq \xi T$ with $\xi \in (-\alpha,g_\alpha )$, for $f$ satisfying Assumption~\ref{assumpt:2.1.1} and for each $S \in \mathbb{R}$, it holds 
	\begin{equation} \begin{aligned} 
	&\lim_{T \to \infty} \mathbb{P}(x^f_{\alpha T}(T) > \xi T - S T^{1/3}) = \prod_{i=0}^n \mathbb{P} \Bigl ( \sup_{\tau \in \mathbb{R}} \{ \mathcal{A}_2(\tau)-g_i(\tau)\} < (c_1^i)^{-1} S \Bigr),
	\end{aligned} \end{equation}
	where $\mathcal{A}_2$ denotes an Airy$_2$ process. 
\end{cor} 

In the setting of Corollary~\ref{cor:thm_2.1}, we have $\alpha_n < 1$. Subsequently, we denote 
\begin{equation}
\Phi_n^{\alpha_0, \dots, \alpha_n}(S) = \prod_{i=0}^n \mathbb{P} \Bigl ( \sup_{\tau \in \mathbb{R}} \{ \mathcal{A}_2(\tau)-g_i(\tau)\}< (c_1^i)^{-1} S \Bigr).
\end{equation}
Note that $\Phi_n^{\alpha_0, \dots, \alpha_n}$ implicitly depends on the functions $g_0, \dots, g_n$ from Assumption~\ref{assumpt:2.1.1}.

While Situation~1 covers the limit distributions in the interior of the region affected by the wall, Situation~2 and Situation~3 deal with the limiting behaviour on its boundary\footnote{Here, we suppose that there exists some $\alpha$ such that the tagged particle is affected by the wall, but $x^f_{\alpha T}(T) \simeq \xi T$ is its original scaling. If this was not the case and the $\alpha T$-th particle was the leftmost particle under a wall influence, then in $\tilde{x}(T)$ there would be a gap between $\tilde{x}_{\alpha T}(T)$ and the particles to its left. Again, the fluctuations of $x^f_{\alpha T}(T)$ could be determined by \eqref{eq:situation_1}.}. Here, the asymptotic behaviour of the tagged particle's fluctuations depends on the initial condition. This makes sense because the limit distribution at the boundary should represent the transition between the limit distributions in the affected region and those of particles not influenced by the wall. 

\paragraph{Situation 2: $\xi = g_\alpha < 1-2\sqrt{\alpha}$. } 

If the tagged particle retains its original scaling and the latter is macroscopically smaller than the scaling in a TASEP with step initial condition, then we observe a decoupling in the limit distribution. Now, this is not due to several wall influences but to the initial condition of the process. 

\begin{thm} \label{th:2.2} 
	Suppose $\xi = g_\alpha < 1-2\sqrt{\alpha}$. Then, it holds
	\begin{equation}
	\begin{aligned}
	\lim_{T \to \infty} \mathbb{P}(x^f_{\alpha T}(T) > \xi T - S T^{1/3}) 
	&= \lim_{T \to \infty} \mathbb{P}(x_{\alpha T}(t) > \xi T - f(T-t)-ST^{1/3} \ \forall t \in [0,T]) \\
	&\hphantom{= } \times \lim_{T \to \infty} \mathbb{P}(\tilde{x}_{\alpha T}(T) > \xi T - S T^{1/3}). 
	\end{aligned} 
	\end{equation}
\end{thm} 

For $g_\alpha > -\alpha$, under Assumption~\ref{assumpt:2.1.1} the first limit on the right hand side again takes the form of $\Phi_n^{\alpha_0,\dots,\alpha_n}(S)$ with $\alpha_n < 1$. We prove Theorem~\ref{th:2.2} in Section~\ref{sect:proof_asymptotics_decoupling}. 

\paragraph{Situation 3: $\xi = g_\alpha = 1-2\sqrt{\alpha}$.} 

If the scaling of the tagged particle matches its original one in the absence of a wall and equals the scaling in a TASEP with step initial condition, then the labels $\alpha T-j$ for which the events at time $T$ in \eqref{eq:thm_finite_time_identity_non_random_IC} have a non-trivial limiting probability include labels in a neighbourhood of $\alpha T$. If we further have a wall influence around time $T$, then we need to determine the joint limit distribution of the tagged particle process $x_{\alpha T}(t)$ on a time interval \emph{and} of a family of particle positions at time $T$. This means that we consider space-like paths like in Figure~\ref{fi:space_like_path}.

We impose additional assumptions on the initial condition of our process, namely the existence of an asymptotic linear growth rate and a convergence under diffusive scaling: 

\begin{assumption} \label{assumpt:2.7} 
	There exists some $d \in [1, \tfrac{1}{\sqrt{\alpha}}]$ such that\footnote{For a half-$d$-periodic initial condition, $g_\alpha = 1-2\sqrt{\alpha}$ holds for $\alpha \leq d^{-2}$, see Section~\ref{sect:half_d_periodic_TASEP}.}
	\begin{equation} \label{eq:additional_assumpt_IC} 
	y_T(\tau) \coloneqq  \frac{x_{1+\hat{c}_2\tau T^{2/3}}^f + d \hat{c}_2\tau T^{2/3}}{c_1 T^{1/3}}, \ \ \tau \geq 0, 
	\end{equation}
	fulfils $y_T(\tau) \geq - \tfrac{1}{2} \tau^2$ for $\tau \in [0,\hat{c}_2^{-1} \alpha T^{1/3}]$ and for all $T$ large enough\footnote{Our arguments apply for $y_T(\tau) \geq - \tfrac{1}{2} \tau^2 - o(1)$ as well.}. The constants are given by $\hat{c}_2 = 2 \alpha^{2/3}(1-\sqrt{\alpha})^{1/3}$ and $c_1 = (1-\sqrt{\alpha})^{2/3} \alpha^{-1/6}$.
	If $d = \tfrac{1}{\sqrt{\alpha}}$, then $y_T$ additionally converges uniformly on compact sets to a piecewise continuous and càdlàg function $y$.
\end{assumption} 

\begin{thm} \label{th:2.3} 
	Suppose $\xi = g_\alpha = 1-2\sqrt{\alpha}$, $f$ satisfies Assumption~\ref{assumpt:2.1.1} with $\alpha_n = 1$ and $\{x_n^f, n \in \mathbb{N}\}$ fulfils Assumption~\ref{assumpt:2.7}. \\
	Case 1: if $d < \tfrac{1}{\sqrt{\alpha}}$, then it holds
	\begin{equation} \label{eq:thm_2.1.3_0}
	\begin{aligned}
	& \lim_{T \to \infty} \mathbb{P}(x^f_{\alpha T}(T) > (1-2\sqrt{\alpha})T - S T^{1/3}) \\
	&= \mathbb{P} \Bigl ( \sup_{\tau \geq 0} \{ \mathcal{A}_2(\tau)-g_n(\tau) \} < c_1^{-1} S \Bigr ) 
	\times  \Phi_{n-1}^{\alpha_0, \dots, \alpha_{n-1}}(S).
	\end{aligned} 
	\end{equation}
	Case 2: if $d = \tfrac{1}{\sqrt{\alpha}}$, then it holds
	\begin{equation} \label{eq:thm_2.1.3_1} 
	\begin{aligned}
	& \lim_{T \to \infty} \mathbb{P}(x^f_{\alpha T}(T) > (1-2\sqrt{\alpha})T - S T^{1/3}) \\ 
	&= \mathbb{P} \Bigl ( \sup_{\tau \geq 0} \{ \mathcal{A}_2(\tau)-g_n(\tau) \} < c_1^{-1} S, \ \sup_{\tau  \leq 0} \{ \mathcal{A}_2(\tau)-\tau^2 - y(-\tau) \} < c_1^{-1} 	S \Bigr ) \\ 
	& \hphantom{=} \times \Phi_{n-1}^{\alpha_0, \dots, \alpha_{n-1}}(S).
	\end{aligned} 
	\end{equation}
\end{thm}

\begin{remark}
	Since we demand $f(0) = 0$, $\xi = 1-2\sqrt{\alpha}$ implies $\alpha_n = 1$ if $f$ fulfils  Assumption~\ref{assumpt:2.1.1}. However, there might not be a wall influence at time $T$, in the sense that $g_n(\tau) = + \infty$ for $\tau > 0$. In this case, the first factor in \eqref{eq:thm_2.1.3_0} becomes $F_{\textup{GUE}}(c_1^{-1}S)$, where $F_{\textup{GUE}}$ denotes the GUE Tracy--Widom distribution function, and the first supremum in \eqref{eq:thm_2.1.3_1} vanishes, as $y(0) = 0$. 
\end{remark} 

In particular, for $x^f_{\alpha T}(T) \simeq (1-2\sqrt{\alpha})T$ and $d < \tfrac{1}{\sqrt{\alpha}}$, we obtain the same limit distribution as in a TASEP with step initial condition and wall constraint, see Theorem~4.7 of \cite{BBF21} and combine it with Theorem~2.5 of \cite{FG24}. 

Theorem~\ref{th:2.3} is proven in Section~\ref{sect:proof_asymptotics_interpolation} with help of some tools from Section~\ref{sect:auxiliary results}. Having the variational formula \eqref{eq:remark_1}, we can also apply them to compute the limit distribution
\begin{equation}
\lim_{T \to \infty} \mathbb{P}(\tilde{x}_{\alpha T}(T) > g_\alpha T - S T^{1/3})
\end{equation}
in the absence of a wall, given that we impose suitable assumptions on $\{x_n^f,n \in \mathbb{N}\}$. In addition to the situation of Assumption~\ref{assumpt:2.7}, we are interested in values $d > \tfrac{1}{\sqrt{\alpha}}$ because they are related to Theorem~\ref{th:2.2}. 

\begin{assumption} \label{assumpt:2.9} 
	There exists some $d > \tfrac{1}{\sqrt{\alpha}}$ such that 
	\begin{equation}
	y_T(\tau) \coloneqq \frac{x^f_{(\alpha-d^{-2})T + \hat{c}_2 \tau T^{2/3}} + d((\alpha-d^{-2})T+\hat{c}_2 \tau T^{2/3})}{c_1 T^{1/3}}, \ \ \tau \in \mathbb{R}, 
	\end{equation}
	fulfils $y_T(\tau) \geq - 2 (1+\sqrt{\alpha}d)^{-2} \tau^2$ for $\tau \in [-\hat{c}_2^{-1} (\alpha-d^{-2}) T^{1/3}, \hat{c}_2^{-1} d^{-2} T^{1/3}]$ and for all $T$ large enough. The constants are $\hat{c}_2 = 2d^{-4/3}(1-d^{-1})^{1/3}$ and $c_1 = (1-d^{-1})^{2/3} d^{1/3}$. Further, $y_T$ converges uniformly on compact sets to a piecewise continuous and càdlàg function $y$. 
\end{assumption} 

\begin{prop} \label{pro:limit_thm_2.2} 
	Recall $\tilde{x}_n(0)=x_n^f$ and suppose $\{x_n^f, n \in \mathbb{N}\}$ fulfils Assumption~\ref{assumpt:2.9}. Then, it holds 
	\begin{equation}
	\begin{aligned} 
	\lim_{T \to \infty} \mathbb{P}(\tilde{x}_{\alpha T}(T) > g_\alpha T - S T^{1/3})
	&= \mathbb{P} \left( \sup_{\tau \in \mathbb{R}} \{ \mathcal{A}_2(\tau) - \tau^2 - y(-\tau) \} < c_1^{-1} S \right)
	\end{aligned} \label{eq:prop_limit_thm_2.2_0} 
	\end{equation}
	with $g_\alpha = 1-d^{-1}-d\alpha$ and $c_1 = (1-d^{-1})^{2/3} d^{1/3}$. 
	
	Suppose $\{x_n^f,n \in \mathbb{N}\}$ fulfils Assumption~\ref{assumpt:2.7}. Then, we have $g_\alpha = 1-2\sqrt{\alpha}$. If $d = \tfrac{1}{\sqrt{\alpha}}$, then it holds
	\begin{equation}
	\begin{aligned} 
	\lim_{T \to \infty} \mathbb{P}(\tilde{x}_{\alpha T}(T) > g_\alpha T - S T^{1/3})
	&= \mathbb{P} \left( \sup_{\tau \leq 0} \{ \mathcal{A}_2(\tau) - \tau^2 - y(-\tau) \} < c_1^{-1} S \right).
	\end{aligned}  \label{eq:prop_limit_thm_2.2_1} 
	\end{equation}
	If $d < \tfrac{1}{\sqrt{\alpha}}$, then it holds
	\begin{equation}
	\lim_{T \to \infty} \mathbb{P}(\tilde{x}_{\alpha T}(T) > g_\alpha T - S T^{1/3}) = F_{\textup{GUE}}(c_1^{-1} S).  \label{eq:prop_limit_thm_2.2_2}  
	\end{equation}
	In both cases, we have $c_1 = (1-\sqrt{\alpha})^{2/3} \alpha^{-1/6}$. 
\end{prop} 

The proof of Proposition~\ref{pro:limit_thm_2.2} can be found in Section~\ref{sect:auxiliary results}.

\begin{remark}
	Proposition~\ref{pro:limit_thm_2.2} does not take initial conditions into account that create a shock in the macroscopic density. Further, as it is the case throughout this work, we restrict ourselves to initial conditions with a rightmost particle. But of course, we expect the limit shape \eqref{eq:prop_limit_thm_2.2_0} also for initial conditions with particles on the whole lattice if they are approximately $d$-periodic with $d > 1$. Choosing the label of the tagged particle large enough, its asymptotic fluctuations in the half-periodic case have the same law like in the fully periodic case, recall \cite{BFS07}. We formulate a convergence result for initial conditions on the whole lattice in Corollary~\ref{cor:prop_2.2}, in accordance with \eqref{eq:prop_limit_thm_2.2_0}. It can be translated to other tagged particles by shifting the labels.
\end{remark} 

\begin{cor} \label{cor:prop_2.2} 
	Consider a TASEP $(\tilde{x}(t),t\geq 0)$ with a deterministic initial condition $\{x_n^f, n \in \mathbb{Z} \}$ with $x_n^f \leq 0$ for $n \in \mathbb{N}$ and $x_n^f > 0$ for $n \in \mathbb{Z}_{\leq 0}$. Let $d > 1$, define $\hat{c}_2$ and $c_1$ as in Assumption~\ref{assumpt:2.9} and suppose that
	\begin{equation} \label{eq:cor_prop_2.2_y} 
	y_T(\tau) \coloneqq \frac{x^f_{\hat{c}_2 \tau T^{2/3}}+d\hat{c}_2 \tau T^{2/3}}{c_1 T^{1/3}}, \ \ \tau \in \mathbb{R}, 
	\end{equation}
	fulfils $y_T(\tau) \geq - 2 (1+d)^{-2} \tau^2$ for $\tau \in [-\hat{c}_2^{-1} (1-d^{-2})T^{1/3}, \hat{c}_2^{-1} d^{-2} T^{1/3}]$ and for all $T$ large enough. Further, suppose that $y_T$ converges uniformly on compact sets to a piecewise continuous and càdlàg function $y$. Then, it holds 
	\begin{equation}
	\begin{aligned} 
	\lim_{T \to \infty} \mathbb{P}(\tilde{x}_{d^{-2} T}(T) > (1-2d^{-1}) T - S T^{1/3})
	&= \mathbb{P}\left( \sup_{\tau \in \mathbb{R}} \{ \mathcal{A}_2(\tau) - \tau^2 - y(-\tau) \} < c_1^{-1} S \right).
	\end{aligned} \label{eq:cor_prop_2.2} 
	\end{equation} 
\end{cor} 

We provide a short proof of Corollary~\ref{cor:prop_2.2} in Section~\ref{sect:auxiliary results}.

\subsection{Large-time asymptotics for TASEP with random initial condition and wall}  \label{sect:main_results_asymptotics_random}

In this paper, we primarily focus on non-random initial conditions. Nevertheless, the finite-time identity from Theorem~\ref{th:finite_time_identity_non_random_IC} holds for all initial conditions, and the results from Section~\ref{sect:main_results_asymptotics} can be generalised to random initial data. Specifically, we observe the following:  

\paragraph{1.} Theorem~\ref{th:2.1} holds true for any random or deterministic initial condition for which a weak law of large numbers is known for $\tilde{x}_{\alpha T}(T)$. This follows directly from its proof. 

\paragraph{2.} Theorem~\ref{th:2.2} remains valid for random initial conditions if, for instance, instead of \eqref{eq:linear_decay_IC}, we assume the existence of some $d \in \mathbb{R}_{\geq 1}$ such that for each $\eps > 0$ there exists a constant $C_\eps > 0$ satisfying 
\begin{equation}
\mathbb{P}(x_n^f \geq -dn-C_\eps \text{ for all } n \in \mathbb{N} ) \geq 1-\eps.
\end{equation}
Indeed, in the first part of the proof of Theorem~\ref{th:2.2}, we can condition on the events described above and then let $\eps \to 0$ afterwards. The rest of the proof stays the same.

\paragraph{3.} Similarly, Theorem~\ref{th:2.3}, Proposition~\ref{pro:limit_thm_2.2} and Corollary~\ref{cor:prop_2.2} extend to random initial conditions. Instead of Assumption~\ref{assumpt:2.7}, we suppose that for some $d \in [1,\tfrac{1}{\sqrt{\alpha}}]$, $y_T$ from \eqref{eq:additional_assumpt_IC} converges weakly with respect to the uniform topology on compact sets to $y$, and that for each $\eps > 0$ there exists a constant $C_\eps > 0$ such that for all $T$ large enough, 
\begin{equation}
\mathbb{P}\bigl(y_T(\tau) \geq - \tfrac{1}{2} \tau^2 - C_\eps \text{ for all } \tau \in [0,\hat{c}_2^{-1} \alpha T^{1/3}]\bigr) \geq 1-\eps.
\end{equation}
Now, the function $y$ may be random. It is again càdlàg and (almost surely) piecewise continuous, and satisfies $y(0)=0$. 

Then, since $y_T$ is independent of the TASEP with step initial condition in Theorem~\ref{th:finite_time_identity_non_random_IC}, the results from Theorem~\ref{th:2.3} and Proposition~\ref{pro:limit_thm_2.2} for $d \leq \tfrac{1}{\sqrt{\alpha}}$ remain valid\footnote{For $d < \tfrac{1}{\sqrt{\alpha}}$, we now use $y_T \Rightarrow y$ with $y(0) = 0$ in the proofs because $y_T(0) \geq - C_\eps$ is not sufficient. Further, we exclude the case $d > \tfrac{1}{\sqrt{\alpha}}$ in Proposition~\ref{pro:limit_thm_2.2} because in the half-Bernoulli case, our standard example for random initial conditions, the randomness emerging from the initial data now operates on a Gaussian scale.}.

In the same way, Corollary~\ref{cor:prop_2.2} holds if we assume that $y_T$ from \eqref{eq:cor_prop_2.2_y} converges weakly with respect to the uniform topology on compact sets to $y$, and that for each $\eps > 0$ there exists a constant $C_\eps > 0$ such that for all $T$ large enough, we have $y_T(\tau) \geq - 2 (1+d)^{-2} \tau^2 - C_\eps$ for all $\tau \in [-\hat{c}_2^{-1} (1-d^{-2})T^{1/3}, \hat{c}_2^{-1} d^{-2} T^{1/3}]$ with probability at least $1-\eps$. 

\begin{remark} 
	In particular, the extended versions of Theorem~\ref{th:2.1}, Theorem~\ref{th:2.2} and Theorem~\ref{th:2.3} apply to the TASEP with half-Bernoulli initial condition and wall. We denote the density of this initial data by $\rho \in (0,1)$. Using standard techniques such as Lundberg's inequality, Donsker's theorem, and Doob's submartingale inequality, one can readily verify that the above assumptions are met. For particles close to the wall, we obtain the same asymptotic fluctuations as in the step initial condition case. If the macroscopic label of the leftmost particle affected by the wall is $\alpha T$ with $\alpha > \rho^2$, then Theorem~\ref{th:2.2} yields a decoupling in the limit distribution. Notably, the scaling of $\tilde{x}_{\alpha T}(T)$ is actually Gaussian in this case, that is, $T^{1/2}$ instead of $T^{1/3}$. Due to the wall influence, fluctuations of $x^f_{\alpha T}(T)$ occur on both scales. If $\alpha = \rho^2$, then Theorem~\ref{th:2.3} applies and, for $d = \rho^{-1}$, the process $(y_T(\tau))$ from \eqref{eq:additional_assumpt_IC} converges weakly to $(\sqrt{2} \mathcal{B}(\tau))$, where $\mathcal{B}$ denotes a standard Brownian motion. 
	
	Furthermore, \eqref{eq:prop_limit_thm_2.2_1} in Proposition~\ref{pro:limit_thm_2.2} and Corollary~\ref{cor:prop_2.2} provide, for example, the one-point limit distributions of the TASEP with half-Bernoulli, Bernoulli, or Bernoulli-flat initial conditions along the characteristic line. The respective limit functions are given by $y(\tau) = \sqrt{2} \mathcal{B}(\tau)$ for $\tau \geq 0$ or $\tau \in \mathbb{R}$ and $y(\tau)=\mathbb{1}_{\tau \geq 0} \sqrt{2} \mathcal{B}(\tau)$ for $\tau \in \mathbb{R}$, where in the two latter cases, $\mathcal{B}$ is a two-sided standard Brownian motion. Naturally, the variational formulas coincide with those obtained in Corollary~2.8 of \cite{CLW16}, see also \cite{CFS16}. Indeed, in \cite{CLW16}, the generalisation from deterministic to random initial conditions was performed using arguments comparable to those above, see Remark~1 of \cite{CLW16}.
\end{remark} 

\subsection{Shocks in the macroscopic density profile} \label{sect:main_shocks} 

We take another look at the framework of Section~\ref{sect:main_results_asymptotics}. Under some circumstances, the limit distribution of the tagged particle's fluctuations takes a product form. We observe that this corresponds to a discontinuity, that is, a shock, in the macroscopic density profile of the process. For $x^f_{\alpha T}(T) \simeq \xi T$, we consider the asymptotic empirical density at $\xi$ and define
\begin{equation}
\rho^f_r(\xi) \coloneqq \liminf_{\eta \downarrow 0} \lim_{T \to \infty} \tfrac{\eta T}{x^f_{(\alpha-\eta)T}(T)-x^f_{\alpha T}(T)},  \ \rho^f_l(\xi) \coloneqq \limsup_{\eta \downarrow 0} \lim_{T \to \infty} \tfrac{\eta T}{x^f_{\alpha T}(T)-x^f_{(\alpha+\eta) T}(T)}.
\end{equation}
Since we suppose that the particles are either not affected by the wall or Assumption~\ref{assumpt:2.1.1} applies for suitable scalings,
\begin{equation}
\frac{\eta T}{x^f_{(\alpha-\eta)T}(T)-x^f_{\alpha T}(T)} \ \text{ and } \ \frac{\eta T}{x^f_{\alpha T}(T)-x^f_{(\alpha+\eta) T}(T)}
\end{equation}
converge in probability to constant values as $T \to \infty$\footnote{By our asymptotic results, the particle fluctuations are in $\mathcal{O}(T^{1/3})$ (for deterministic initial conditions). Also a weak law of large numbers is readily obtained by the formula \eqref{eq:thm_finite_time_identity_non_random_IC} and by one-point estimates for TASEP with step initial condition, see \cite{BBF21}.}. We compute $\rho^f_r(\xi)$ and $\rho^f_l(\xi)$ by taking $\eta \downarrow 0$ on these limits. See Figure~\ref{fi:example_limit_distr_step} and Figure~\ref{fi:example_limit_distr_periodic} for examples. 

\begin{lem} \label{lemma_density_wall_influence_times}
	Suppose $x^f_{\alpha T}(T) \simeq \xi T$ and $f$ fulfils Assumption~\ref{assumpt:2.1.1} with parameters $\alpha_0 < \dots < \alpha_n$. Then, it holds
	\begin{equation} 
	\rho^f_r(\xi) = \sqrt{ \frac{\alpha}{\alpha_0}} \ \text{ and } \  \rho^f_l(\xi) \leq \sqrt{\frac{\alpha}{\alpha_n}}. 
	\end{equation}   
	In the definition of $\rho^f_r(\xi)$, the limit for $\eta \downarrow 0$ exists. If $x^f_{\alpha T}(T)$ is in the interior of the region affected by the wall and Assumption~\ref{assumpt:2.1.1} is fulfilled with $\eps \leq \alpha_n$, then this is also the case for $\rho^f_l(\xi)$, and we have $ \rho^f_l(\xi) = \sqrt{\alpha \alpha_{n \vphantom{0}}^{-1}}$.
\end{lem} 

Lemma~\ref{lemma_density_wall_influence_times} is proven in Section~\ref{sect:proof_density_shock}. The assumption $\eps \leq \alpha_n$ is technical and emerges from series expansions in the proof. 

\paragraph{Shocks at points of several wall influences.} The limit distribution of the tagged particle's fluctuations takes a product form whenever there are several time regions where the wall influences the particle, that is, if $n > 0$ in Lemma~\ref{lemma_density_wall_influence_times}. The lemma implies that indeed, the macroscopic empirical density has a discontinuity. 

\paragraph{Shocks at the boundary of the region affected by the wall.} If $x^f_{\alpha T}(T) \simeq \xi T$ with $\xi =g_\alpha < 1-2\sqrt{\alpha}$ and there is a wall influence at a time $\alpha_0 T$ with $\alpha_0 \in (\alpha,1)$, then Theorem~\ref{th:2.2} provides a product limit distribution. Lemma~\ref{lemma_density_wall_influence_times} implies $\rho^f_r(\xi) \geq \sqrt{\alpha \alpha_0^{-1}}$. However, to the left of $\xi T$, the particles are not affected by the wall until time $T$. Thus, the macroscopic density there is the same like in $\tilde{x}(T)$. Since TASEP is attractive and we consider initial conditions with no particles to the right of the origin, the density around the $\alpha T$-th particle is bounded from above by the density around the $\alpha T$-th particle in a TASEP with step initial condition, that is, $\sqrt{\alpha}$. This implies $\rho^f_l(\xi) < \rho^f_r(\xi)$. 

\section{TASEP with half-$d$-periodic initial condition and wall} \label{sect:half_d_periodic_TASEP}

Special cases of the processes discussed in Section~\ref{sect:main_results} are TASEPs with half-$d$-periodic initial conditions and wall constraints, where $d \in \mathbb{R}_{\geq 1}$. In this section, we take a detailed look at them to understand when the different situations described in Section~\ref{sect:main_results_asymptotics} occur. We recall that such a process, denoted by $(x^f(t),t \geq 0)$, has the following properties: 
\begin{equation}
x^f_n(0) = - \lfloor d (n-1) \rfloor, \ n \in \mathbb{N}, \ \text{  and  } \ x^f_1(t) \leq f(t), \ t \geq 0.
\end{equation}
We again let $(x(t),t \geq 0)$ be a TASEP with step initial condition and denote by $(\tilde{x}(t),t \geq 0)$ a TASEP with half-$d$-periodic initial condition. 

For $\alpha \in (0,1)$, we have $\tilde{x}_{\alpha T}(T) \simeq g_\alpha T$  with
\begin{equation}
g_\alpha  = \begin{cases}
1-2\sqrt{\alpha} \ & \text{if } \alpha \in (0,d^{-2}], \\ 1-d^{-1}-d\alpha \ & \text{if } \alpha \in (d^{-2},1). 
\end{cases}
\end{equation}
This scaling readily emerges from hydrodynamic theory: the macroscopic density of $\tilde{x}(T)$ is given by \cite{Li99}
\begin{equation} \label{eq:3_density} 
\tilde{\rho}(x,T) = \begin{cases}\tfrac{1}{d} \ & \text{if} \ x \leq (1-\tfrac{2}{d})T, \\ \tfrac{1}{2}(1-\tfrac{x}{T}) \ & \text{if} \ (1-\tfrac{2}{d})T<x<T, \\ 0 \ & \text{if} \ x \geq T. \end{cases} 
\end{equation}

By Proposition~\ref{pro:limit_thm_2.2}, see also \cite{BFP06,BFPS06,BFS07,FO17}, the fluctuations of $\tilde{x}_{\alpha T}(T)$ show the following asymptotic behaviour: 
\begin{equation} \label{eq:3_limits} 
\begin{aligned}
\lim_{T \to \infty} \mathbb{P}(\tilde{x}_{\alpha T}(T) > g_\alpha T - S T^{1/3}) = \begin{cases} 
F_{\textup{GUE}}((1-\sqrt{\alpha})^{-2/3} \alpha^{1/6} S)  & \text{if } \alpha \in (0,d^{-2}), \\
F_{2 \to 1;0}((1-\sqrt{\alpha})^{-2/3} \alpha^{1/6} S)  & \text{if } \alpha = d^{-2}, \\ 
F_{\textup{GOE}}(2^{2/3} (1-d^{-1})^{-2/3} d^{-1/3} S)   &\text{if } \alpha > d^{-2}.
\end{cases} 
\end{aligned} 
\end{equation}
Here, $F_{\textup{GOE}}$ is the GOE Tracy--Widom distribution function and $F_{2 \to 1;0}$ denotes the distribution of the Airy$_{2 \to 1}$ process in $0$. In addition to Proposition~\ref{pro:limit_thm_2.2}, we used the variational formulas from Corollary~1.3 of \cite{Jo03b} and Theorem~1 of \cite{QR13b}. \\

We now focus on the process with a wall. The maximal effect on the particles is produced by a wall that remains at the origin for all times, that is, $f \equiv 0$. The particles pile up behind this wall and form a block of density $1$. The largest $\alpha$ such that $x^f_{\alpha T}(T)$ is influenced by this wall is such that $g_\alpha T$ is at the left boundary of the block, meaning $g_\alpha  = -\alpha$. Thus, $\alpha = d^{-1}$. We deduce that for any wall under consideration, particles with labels $\alpha T, \alpha \in (d^{-1},1)$ are not affected. From now on, we suppose $\alpha \in (0,d^{-1}] \cap (0,1)$.  

As before, we are interested in the case $x^f_{\alpha T}(T) \simeq \xi T$ with $\xi > -\alpha$. Notice that for $\alpha \in (0,d^{-1})$, $\xi \leq -\alpha$ reduces to $\xi = -\alpha$ because, by previous reasoning, the $d^{-1}T$-th particle always has the scaling $g_{d^{-1}} T = - d^{-1} T$. In particular, between $x^f_{d^{-1}T}(T)$ and $x^f_{\alpha T}(T)$, the macroscopic density equals $1$. Indeed, Theorem~\ref{th:2.1} implies that, in this case, the $\alpha T$-th particle asymptotically ends up to the right of or at $-\alpha T$. If $f \equiv 0$, there are no asymptotic fluctuations.\\

Our results in Section~\ref{sect:main_results_asymptotics} provide the limit distributions of the tagged particle's fluctuations, depending on whether it is located in the interior of the region affected by the wall or ends up on its boundary: 

\paragraph{1.} For $\xi \in (-\alpha,g_\alpha )$, we obtain the same convergence results like for TASEP with step initial condition and wall constraint. This is a consequence of Theorem~\ref{th:2.1}.

\paragraph{2.} For $\xi = g_\alpha$ and $\alpha \in (d^{-2},d^{-1})$, Theorem~\ref{th:2.2} applies. If $f$ satisfies Assumption~\ref{assumpt:2.1.1}, we obtain for each $S \in \mathbb{R}$:
\begin{equation}
\begin{aligned}
\lim_{T \to \infty} \mathbb{P}(x^f_{\alpha T}(T) > (1-d^{-1}-d\alpha) T - S T^{1/3}) =  F_{\textup{GOE}}(2^{2/3}c_1^{-1} S) \times \Phi_n^{\alpha_0, \dots, \alpha_n}(S)
\end{aligned} 
\end{equation}
with $c_1 =(1-d^{-1})^{2/3} d^{1/3}$. In the special case $f \equiv 0$, Theorem~\ref{th:2.2} confirms that the $d^{-1}T$-th particle asymptotically shows GOE fluctuations to the left and no fluctuations to the right.  

\paragraph{3.} For $\xi = g_\alpha$ and $\alpha \in (0,d^{-2}]$, Theorem~\ref{th:2.3} applies. If $f$ fulfils Assumption~\ref{assumpt:2.1.1} with $\alpha_n =1$, then $\lim_{T \to \infty} \mathbb{P}(x^f_{\alpha T}(T) > (1-2\sqrt{\alpha})T - S T^{1/3})$ equals 

\begin{equation} \label{eq:cor_2.1.3_0}
\mathbb{P} \Bigl ( \sup_{\tau \geq 0} \{ \mathcal{A}_2(\tau)-g_n(\tau) \} < (c_1^n)^{-1} S \Bigr )  \times  \Phi_{n-1}^{\alpha_0, \dots, \alpha_{n-1}}(S) 
\end{equation}
if $\alpha < d^{-2}$ and 
\begin{equation} \label{eq:cor_2.1.3_1} 
\mathbb{P} \Bigl ( \sup_{\tau \geq 0} \{ \mathcal{A}_2(\tau)-g_n(\tau) \} < c_1^{-1} S, \ \sup_{\tau  \leq 0} \{ \mathcal{A}_2(\tau)-\tau^2 \} < c_1^{-1} 	S \Bigr )  \times  \Phi_{n-1}^{\alpha_0, \dots, \alpha_{n-1}}(S)
\end{equation}
if $\alpha = d^{-2}$. 

We conclude this section with an example. 

\begin{figure}[t!]
	\centering  \includegraphics[scale=1]{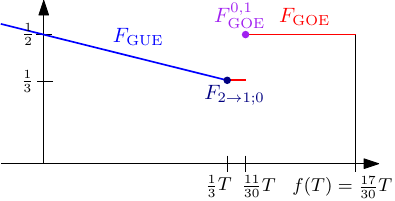}
	\caption{The macroscopic density profile of $x^{\textup{step},f}(T)$ in Example~\ref{example}. We denote $F_{\textup{GOE}}^{0,1}(s) = F_{\textup{GOE}}(c_0 s ) \times  F_{\textup{GOE}}(c_1 s)$ for constants $c_0,c_1 > 0$ to be specified below.} 
	\label{fi:example_limit_distr_step}
\end{figure}

\begin{figure}[t!]
	\centering
	\includegraphics[scale=0.7]{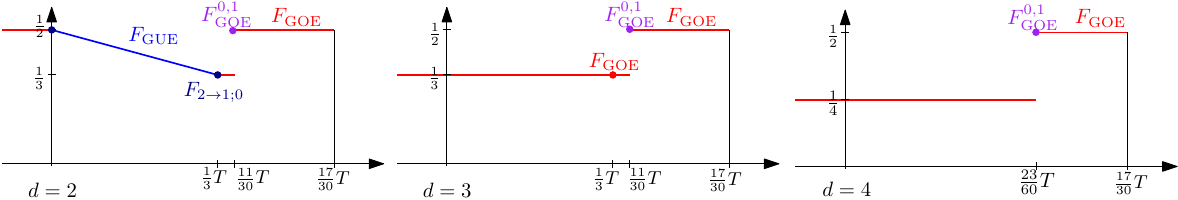}
	\caption{The macroscopic density profiles of $x^f(T)$ in Example~\ref{example} for $d=2,3,4$. We have $f(T)=\tfrac{17}{30}T$. In the interior of the regions affected by the wall, the limiting fluctuations match those in the step initial condition case (Situation 1 in Section~\ref{sect:main_results_asymptotics}). For $d=2,3$, the limit distribution near the boundary of the regions at $\tfrac{1}{3}T$ interpolates (Situation 3). For $d=4$, the boundary is located at $\tfrac{23}{60}T$, and the limiting distribution is a product of two GOE Tracy--Widom distributions (Situation 2). }
	\label{fi:example_limit_distr_periodic} 
\end{figure}

\begin{example} \label{example} We consider the piecewise linear function 
	\begin{equation}
	f(t) = \begin{cases} \tfrac{2}{3}t, \ &  \ t \in [0,0.35T), \\ \tfrac{1}{15}T+\tfrac{1}{2}t, \ & \ t \in [0.35T,T]. \end{cases}
	\end{equation}
	By \cite{BBF21,FG24}, the one-point limit distributions in a TASEP $(x^{\textup{step},f}(t),t\geq 0)$ with step initial condition and moving wall $f$ are:
	
	\begin{equation} \label{eq:example_0}
	\lim_{T \to \infty} \mathbb{P}(x^{\textup{step},f}_{\alpha T}(T) > \xi T - S T^{1/3})  = \begin{cases} F_{\textup{GOE}}( (\tfrac{2}{\alpha})^{1/3}  S),  & \alpha \in (0,\tfrac{1}{10}), \\ 
	F_{\textup{GOE}}( (\tfrac{2}{\alpha})^{1/3} S) F_{\textup{GOE}}((3\alpha)^{-1/3} S),   & \alpha =\tfrac{1}{10}, \\
	F_{\textup{GOE}}((3\alpha)^{-1/3} S),  & \alpha \in (\tfrac{1}{10},\tfrac{1}{9}), \\ 
	F_{2 \to 1;0}(2^{-2/3} 3^{1/3} S), & \alpha = \tfrac{1}{9}, \\ 
	F_{\textup{GUE}}(c_1(\alpha)^{-1}S),  & \alpha \in (\tfrac{1}{9},1), \end{cases} 
	\end{equation}
	where $c_1(\alpha) = \alpha^{-1/6}(1-\sqrt{\alpha})^{2/3}$ and $\xi = \tfrac{17}{30}-2\alpha$ for $\alpha \in (0,\tfrac{1}{10}]$, $\xi = \tfrac{2}{3}-3\alpha$ for $\alpha \in [\tfrac{1}{10},\tfrac{1}{9}]$ and $\xi = 1-2\sqrt{\alpha}$ else. The asymptotic fluctuations in the different regions of the macroscopic density profile are in accordance with KPZ theory, see Figure~\ref{fi:example_limit_distr_step}.
	
	We compare \eqref{eq:example_0} to the  half-$d$-periodic case for $d=2$, $d=3$ and $d \geq 4$. The corresponding density profiles are depicted in Figure~\ref{fi:example_limit_distr_periodic}.
	
	For $d=2$, the leftmost particle affected by the wall is contained in the rarefaction fan region of the original density profile \eqref{eq:3_density}. Hence, for $\alpha \in (0,\tfrac{1}{9}]$, we obtain the same limit distributions as in \eqref{eq:example_0}. For $ \alpha > \tfrac{1}{9}$, we find the original limit distributions from \eqref{eq:3_limits}. 
	
	For $d=3$, we obtain the same limit distributions as in \eqref{eq:example_0} for $\alpha \in (0,\tfrac{1}{9})$. In this case, the $d^{-2}T$-th particle is at the boundary of the region affected by the wall, and
	\begin{equation}
	\lim_{T \to \infty} \mathbb{P}(x^{f}_{d^{-2} T}(T) > (1-2d^{-1}) T - S T^{1/3}) =F_{\textup{GOE}}(3^{1/3}S).
	\end{equation}
	For $\alpha > d^{-2}$, we get convergence to $F_{\textup{GOE}}(3^{1/3}S)$ by \eqref{eq:3_limits}.
	
	Lastly, for $d \geq 4$, we find for $\alpha_d = \tfrac{13d-30}{30(d-2)d}$:
	\begin{equation} \begin{aligned} 
	\lim_{T \to \infty} \mathbb{P}(x^{f}_{\alpha T}(T) > \xi T - S T^{1/3})  &= \begin{cases} F_{\textup{GOE}}((\tfrac{2}{\alpha})^{1/3}  S), \hphantom{  F_{\textup{GOE}}(2^{2/3}c_1(d^{-2})^{-1}S)} \  \alpha \in (0,\alpha_d), \\
	F_{\textup{GOE}}((\tfrac{2}{\alpha})^{1/3} S) F_{\textup{GOE}}(2^{2/3}c_1(d^{-2})^{-1}S), \  \alpha = \alpha_d, \\ 
	F_{\textup{GOE}}(2^{2/3}c_1(d^{-2})^{-1}S), \hphantom{F_{\textup{GOE}}((\tfrac{2}{\alpha})^{1/3} S)} \ \alpha > \alpha_d,
	\end{cases} 
	\end{aligned} \end{equation}
	where $\xi = \tfrac{17}{30}-2\alpha$ for $\alpha \in (0,\alpha_d)$ and $\xi = g_\alpha $ else. Here, the product limit distribution is not due to several wall influences but emerges from a decoupling at the boundary of the region affected by the wall. 
\end{example}

\section{Proof of the finite-time identity} \label{sect:proof_finite_time_identity} 

In this section, we prove Theorem~\ref{th:finite_time_identity_non_random_IC}. For this purpose, we make use of multi-species TASEPs and the colour-position symmetry, as reviewed in Section~3 of \cite{BBF21}. It is useful to describe these processes in terms of their particle configurations.

In a continuous-time multi-species TASEP, each particle is assigned a colour in $\mathbb{Z} \cup \{+\infty\}$. A particle configuration $\eta : \mathbb{Z} \to \mathbb{Z} \cup \{+\infty\}$ maps each position in $\mathbb{Z}$ to the colour of the particle located there. Lower colours have priority over higher colours, and the value $+\infty$ is interpreted as a hole. 

Since it will be used in the proof, we recall the standard construction of the process as described in \cite{BBF21}, which guarantees well-definedness by the usual arguments \cite{Har72,Har78,Hol70,Lig72}. The multi-species TASEP starting at a configuration $\eta_0$ is constructed with help of a family of jointly independent Poisson processes $\{\mathcal{P}_z, z \in \mathbb{Z} \}$ with rate $1$ that are independent of $\eta_0$ as well. At each time when a Poisson process $\mathcal{P}_z$ has an event, a \textit{swap operator} is applied to the configuration of the process. 
We define these operators by 
\begin{equation}
W_{(z,z+1)} \eta = \begin{cases} \eta \ & \text{if} \ \eta(z) \geq \eta(z+1), \\ \sigma_{(z,z+1)} \eta \ & \text{if} \ \eta(z) < \eta(z+1), \end{cases}
\end{equation}
with 
\begin{equation}
(\sigma_{(z,z+1)}\eta)(i) = \begin{cases} \eta(z+1), \ & i=z, \\ \eta(z), \ &i=z+1, \\ \eta(i), \ & i \in \mathbb{Z} \setminus \{z,z+1\}. \end{cases} 
\end{equation}
In words, $W_{(z,z+1)}$ exchanges the particles at the sites $z$ and $z+1$ if beforehand, the particle at $z$ has a lower colour. Several TASEPs are coupled by \emph{basic coupling} if they are constructed using the same family $\{ \mathcal{P}_z, z \in \mathbb{Z}\}$.  

TASEPs with less colours can be viewed as marginals of multi-species TASEPs by bundling particles with colours in subsets forming a partition of $\mathbb{Z} \cup \{+\infty\}$ into one respective type. The partition must be such that a prioritisation of the groups among each other is maintained. For example, in a multi-species TASEP with \textit{packed} initial condition, meaning $\eta_0(z) = z$ for each $z \in \mathbb{Z}$, one can view colours $\leq 0$ as particles and colours $> 0$ as holes, and obtains a single-species TASEP with step initial condition as marginal. 

We consider multi-species TASEPs whose configurations correspond to permutations of the integers. The following result is known as colour-position symmetry: 

\begin{prop} \label{pro:colour_position_symmetry} 
	Let $\textup{id}:\mathbb{Z} \to \mathbb{Z}$ be the identity bijection and denote by $\textup{inv}$ the map that takes the inverse of a permutation. Then, for any $k \in \mathbb{N}$ and integers $z_1, \dots, z_k \in \mathbb{Z}$, it holds 
	\begin{equation}
	W_{(z_k,z_k+1)} \dots W_{(z_1,z_1+1)} \textup{id} = \textup{inv} (W_{(z_1,z_1+1)} \dots W_{(z_k,z_k+1)} \textup{id}).
	\end{equation}
\end{prop} 

Here, we restated Proposition~3.3 of \cite{BBF21}. It was first proven in a probabilistic setting as Lemma~2.1 of \cite{AHR09} and was subject to further generalisations, see \cite{AAV11,BB19} and \cite{Buf20} for the algebraic background. 

Recalling the construction from the previous paragraph, Proposition~\ref{pro:colour_position_symmetry} tells us the following: let $(\eta_t,t\geq 0)$ be a multi-species TASEP with initial condition $W_{s_1} \dots W_{s_m} \text{id}$ for $m \in \mathbb{N}_0$ and pairs $s_i = (z_i, z_i+1), i = 1, \dots, m$. We denote by $(\hat{\eta}_t, t \in [0,T])$ a multi-species TASEP with $\hat{\eta}_0 = \textup{id}$ that is only defined on the time interval $[0,T]$ and after its evolution, $W_{s_m} \dots W_{s_1}$ is applied to its configuration. Then, we have
\begin{equation}
\textup{inv}(\eta_T) \overset{(d)}{=} \hat{\eta}_T.
\end{equation}
The statement remains true if we impose a wall constraint, see also Proposition~3.4 of \cite{BBF21}. If for $\eta^f_t$, jumps are suppressed at sites to the right of $f(t)$, then 
\begin{equation} \label{eq:colour_position_symmetry_wall}
\text{inv}(\eta^f_T) \overset{(d)}{=} \hat{\eta}_{f,T},
\end{equation}
where for $\hat{\eta}_{f,t}$, jumps are suppressed at sites to the right of $f(T-t)$. Apart from the walls, the processes $(\eta^f_t,t\geq 0)$ and $(\hat{\eta}_{f,t}, t \in [0,T])$ have the same properties as above. 

\paragraph{Notation.} Before proceeding with the proof of Theorem~\ref{th:finite_time_identity_non_random_IC}, we summarise the notation for reference. 
To simplify, we will denote the initial condition $\{x_n^f, n \in \mathbb{N}\}$ by $\{u_n, n \in \mathbb{N}\}$. Then, $(x^f(t),t\geq 0)$ represents the TASEP with wall $f$ and initial condition $x_n^f(0)=u_n$. Below, we relate $x_n^f(T)$ to a TASEP $(y(t),t \in [0,T])$ with step initial condition and jumps only allowed at sites $\geq s+1-f(T-t)$. This process is coupled to another TASEP with step initial condition, denoted by $(x(t),t\geq 0)$, without wall constraint.

\begin{proof}[Proof of Theorem~\ref{th:finite_time_identity_non_random_IC}]
	Throughout this proof, we consider $\{x_n^f, n \in \mathbb{N} \}$ as a fixed non-random initial condition. However, our combinatorial arguments also extend to the random case. If one prefers to avoid working with random objects, one can observe that $x_n^f(T)$ only depends on the positions of the first $n$ particles in the initial configuration. For this reason, the law of total probability implies 
	\begin{equation}
	\begin{aligned} 
	&\mathbb{P}(x_n^f(T) > s) \\ & = \sum_{(y_n,\dots,y_1) \in I} \mathbb{P}(x_n^f(T) > s \ | \ x_1^f = y_1, \dots, x_n^f=y_n) \mathbb{P}(x_1^f = y_1, \dots, x_n^f=y_n),
	\end{aligned} 
	\end{equation}
	where $I = \{y_n < \dots < y_1 \leq 0 : \ \mathbb{P}(x_1^f=y_1, \dots, x_n^f=y_n) > 0\}$. The conditional probabilities in the sum can be treated in the same way as for deterministic initial conditions. 
	
	To simplify the notation, we denote $x_n^f$ by $u_n$ in the following.
	
	For $n=1$, the result follows from Remark~3.2 of \cite{BBF21} because $(x^f_1(t),t \geq 0)$ behaves like the first particle in a TASEP with step initial condition shifted by $u_1$ and with wall $f$. From now on, let $n \geq 2$. 
	
	We first couple $(x^f(t),t \geq 0)$ with a multi-species TASEP $(\eta^f_t, t \geq 0)$. On the latter, we impose $\eta^f_0 = \pi (\text{id})$ and allow jumps involving only positions $\leq f(t)$. The permutation $\pi$ is defined as follows: 
	we fix $k \in \{0, \dots, n-2\}$ such that \begin{equation} u_{n-k} < u_n + n - 1 \leq u_{n-k-1}. \end{equation}
	In order to match the initial condition of $(x^f(t),t\geq 0)$ on sites $\geq u_n$, we want $\pi$ to swap the particles in $\{u_n+1, \dots, u_n+n-1\} \setminus \{u_{n-1}, \dots, u_{n-k}\}$ with those in $\{u_{n-k-1}, \dots, u_1\}$ as depicted in Figure~\ref{fi:permutation}. We set
	\begin{equation}
	\pi = \pi_k \pi_{k-1} \dots \pi_0
	\end{equation}
	with 
	\begin{equation}
	\begin{aligned}
	&\pi_j = \pi_j^{u_{n-j-1}-u_{n-j}-2} \dots \pi_j^0, \ \ j = 0, \dots, k-1, \\
	& \pi_k = \pi_k^{u_n+n-1-u_{n-k}-1} \dots \pi_k^0
	\end{aligned}
	\end{equation}
	and $\pi_j^i$ exchanging $u_{n-j} + 1 + i$ and $u_{1+u_{n-j}-u_n-j+i}$. Notice that $u_{n-j}-u_n-j$ is the fixed number of empty sites between $u_n$ and $u_{n-j}$ in the initial condition. We set $\pi_j = \text{id}$ if $u_{n-j-1} = u_{n-j}+1$. Further, we have $\pi_k^{u_n+n-1-u_{n-k}-1} = \text{id}$ if $u_n+n-1=u_{n-k-1}$. The permutations are applied from right to left. 
	We can decompose them into transpositions: 
	\begin{equation}
	\begin{aligned}
	\pi_j^i = & (u_{1+u_{n-j}-u_n-j+i}-1,u_{1+u_{n-j}-u_n-j+i}) \dots (u_{n-j} + 1 + i, u_{n-j} + 2 + i) \\ & \dots (u_{1+u_{n-j}-u_n-j+i}-1,u_{1+u_{n-j}-u_n-j+i}).
	\end{aligned}
	\end{equation}
	It holds $\pi = \pi^{-1}$. In particular, for $W_{s_1}, \dots, W_{s_m}$ being the swap operators corresponding to $\pi$, we have $\pi(\text{id}) = W_{s_1} \dots W_{s_m}(\text{id}) =  W_{s_m} \dots W_{s_1}(\text{id})$. From now on, we only write $\pi$ respectively $\pi_j$ and so forth, implicitly meaning the swap operators. 
	
	\begin{figure}[t!]
		\centering
		\includegraphics[scale=0.9]{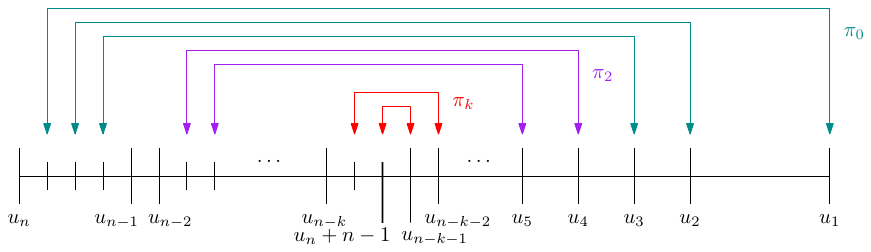}
		\caption{The permutation $\pi$. In this example, $\pi_1 = \text{id}$ and $u_n+n-1 < u_{n-k-1}$. We have set $u_n=x_n^f$. }
		\label{fi:permutation}
	\end{figure}
	
	We couple $(x^f(t),t\geq 0)$ and $(\eta^f_t,t\geq 0)$ by viewing colours $\leq u_n+n-1$ as particles and colours $> u_n+n-1$ as holes. Then, $\eta^f_0$ corresponds to $x^f(0)$ on sites $\geq u_n$ and we obtain 
	\begin{equation}
	\begin{aligned}
	\mathbb{P}(x^f_n(T) > s) = & \mathbb{P} ( \exists \text{ at least } n \text{ numbers } i \leq u_n+n-1: \ \text{inv}(\eta^f_T)(i) > s) \\
	=& \mathbb{P} ( \exists \text{ at least } n \text{ numbers } i \leq u_n+n-1: \ \hat{\eta}_{f,T}(i) > s).
	\end{aligned}
	\end{equation}
	The second identity is due to the colour-position symmetry \eqref{eq:colour_position_symmetry_wall}. Here, $\hat{\eta}_{f,T}$ is a configuration of a multi-species TASEP starting from $\text{id}$ with the following properties: we first let the process evolve up to time $T$, allowing jumps if they only involve positions $\leq f(T-t)$. Afterwards, we apply (the swap operators corresponding to) $\pi $ to the permutation. 
	
	Let $\hathat{\eta}_{f,T}$ be the configuration before $\pi$ is applied, i.e. $\pi(\hathat{\eta}_{f,T}) = \hat{\eta}_{f,T}$. Our main task is to write 
	\begin{equation}
	E \coloneqq \{\exists \text{ at least } n \text{ numbers } i \leq u_n+n-1: \ \hat{\eta}_{f,T}(i) > s \}
	\end{equation}
	in terms of $\hathat{\eta}_{f,T}$. From now on, we identify colours $> s$ as holes and colours $\leq s$ as particles. We define 
	\begin{equation}
	A \coloneqq \bigcap_{j=0}^{n-1} \{ \exists \text{ at least } n-j \text{ numbers } i \leq u_{1+j}: \ \hathat{\eta}_{f,T}(i) > s\}
	\end{equation}
	and show below that $E=A$. Given this identity, we couple the multi-species TASEP $(\hathat{\eta}_{f,t},t \in [0,T])$, with $\hathat{\eta}_{f,0} = \text{id}$, with a single-species TASEP $(y(t),t \in [0,T])$ as follows: we identify colours $>s$ as holes, apply the particle-hole duality, and perform the coordinate shift $z \mapsto s-z+1$. The resulting TASEP $(y(t),t \in [0,T])$ has step initial condition and jumps are only allowed at positions $\geq s+1-f(T-t)$. We obtain 
	\begin{equation}
	\begin{aligned}
	\mathbb{P}(x^f_n(T) > s) = \mathbb{P}(A) = \mathbb{P}(\forall j \in [n-1]: y_{n-j}(T) \geq s+1-u_{1+j}).
	\end{aligned} 
	\end{equation}
	We denote by $(x(t),t\geq 0)$ a TASEP with step initial condition and without wall constraint, and couple it with $(y(t),t \in [0,T])$ by basic coupling, in the time interval $[0,T]$. We claim: 
	\begin{equation} \label{eq:proof_general_IC_coupling_y_and_x}
	\begin{aligned}
	&\mathbb{P}(\forall j \in [n-1]: y_{n-j}(T) \geq s+1-u_{1+j}) \\
	&= \mathbb{P}(\forall t \in [0,T]: x_n(t) \geq s+1-f(T-t), \ \forall j \in [n-1]: x_{n-j}(T) \geq s+1-u_{1+j}).
	\end{aligned} 
	\end{equation}
	First suppose $x_n(t) \geq s+1-f(T-t)$ for all $t \in [0,T]$ and $x_{n-j}(T) \geq s+1-u_{1+j}$ for all $j \in [n-1]$. By the first inequality, it holds $y_{n-j}(t) = x_{n-j}(t)$ for all $t \in [0,T]$, $j \in [n-1]$. In particular, we find $y_{n-j}(T) = x_{n-j}(T) \geq s+1-u_{1+j}$. 
	
	Next, suppose $x_n(t_0) < s+1-f(T-t_0)$ for some $t_0 \in [0,T]$. Then, $y_n(t_0) \leq x_n(t_0)$ implies $y_n(T) = y_n(t_0) < s + 1 - f(T-t_0) \leq s+1 \leq s+1-u_1$ since $u_1 \leq 0$. 
	
	Finally, suppose $x_{n-j}(T) < s+1-u_{1+j}$ for some $j \in [n-1]$. In this case, we obtain $y_{n-j}(T) \leq x_{n-j}(T) < s+1-u_{1+j}$. 
	
	Thus, \eqref{eq:proof_general_IC_coupling_y_and_x} holds true and we deduce \eqref{eq:thm_finite_time_identity_non_random_IC}. 
	
	In the remainder of the proof, we demonstrate that $E=A$. 
	
	\paragraph{Step 1: $A \subseteq E$.} 
	Suppose $A$ occurs. Denote by $\ell$ the number of holes in $\hathat{\eta}_{f,T}$ at sites $\leq u_n+n-1$. If $\ell \geq n$, then $E$ occurs since the numbers of holes at sites $\leq u_n+n-1$ cannot decrease by an application of swap operators. Thus, we need to consider the case $\ell < n$. 
	
	\paragraph{Claim 1:} Suppose for some $i,j$ such that $\pi_j^i \neq \text{id}$ and for some $x \in \{u_{n-j}+1+i, \dots, u_{1+u_{n-j}-u_n-j+i}-1\}$, in $(\pi_j^{i-1} \dots \pi_0^0)(\hathat{\eta}_{f,T})$ (or $(\pi_{j-1}^{u_{n-j} - u_{n-j+1}-2} \dots \pi_0^0)(\hathat{\eta}_{f,T})$ or $\hathat{\eta}_{f,T}$) there exist a hole at a site in $\{x+1, \dots,  u_{1+u_{n-j}-u_n-j+i}\}$ and a particle at a site in $\{u_{n-j}+1+i, \dots, x \}$. Then, the application of $\pi_j^i$ increases the number of holes in $\{ u_{n-j}+1+i, \dots, x\}$ by $1$. Further, $\pi_j^i$ moves a hole to site $u_{n-j}+1+i$ and a particle to site $u_{1+u_{n-j}-u_n-j+i}$. 
	
	Claim 1 follows from Lemma~\ref{lem:permutation_increases_number_of_holes}. In addition, $A$ implies 
	\begin{itemize}
		\item There are at least $n-j-\ell$ holes in $\{u_n+n, \dots, u_{1+j}\}$ for $j=0, \dots, j^*$, where 
		\begin{equation}
		j^* = \begin{cases} n-k-2 &\text{  if  } u_n+n-1 < u_{n-k-1}, \\n-k-3 &\text{  if  } u_n+n-1 = u_{n-k-1}. \end{cases}
		\end{equation}
		\item There are at least $ u_n - u_{n-j} +n-\ell+j$ particles in $\{u_{n-j}+1, \dots, u_n+n-1\}$ for $j=0, \dots, k$.
	\end{itemize}
	
	Combining this with Claim 1, we gain information on the number of holes and particles in certain regions after applying some of our swap operators to $\hathat{\eta}_{f,T}$:
	
	\paragraph{Claim 2:} Let $r \in \{0, \dots, k-1\}$ with $\pi_r \neq \text{id}$, $i \in \{0, \dots, u_{n-r-1}-u_{n-r}-2\}$ and $\ell < n-u_{n-r}+u_n+r-i$. Then, in $(\pi_r^i \dots \pi_0^0)(\hathat{\eta}_{f,T})$, there are 
	\begin{itemize} \item[(i)] $\ell - u_n+u_{n-r}-r+i+1$ holes at sites $\leq u_n+n-1$;
		\item[(ii)] at least $n-j-\ell$ holes in $\{u_n+n, \dots, u_{1+j} \}$ for $j = u_{n-r}-u_n-r+i+1, \dots, j^*$; 
		\item[(iii)] at least $u_n - u_{n-r}+n-\ell+r-i-1$ particles in $\{u_{n-r}+i+2, \dots, u_n+n-1\}$; 
		\item[(iv)] at least $u_n - u_{n-j} +n-\ell+j$ particles in $\{u_{n-j}+1, \dots, u_n+n-1\}$ for $j = r+1, \dots, k$. 
	\end{itemize}
	
	We decompose the proof of Claim 2 as follows: 
	\begin{itemize}
		\item[(a)] (i)--(iii) hold true for $\pi_r^0(\hathat{\eta}_{f,T})$ with $r$ minimal such that $\pi_r \neq \text{id}$.
		\item[(b)] If $\pi_r \neq \text{id}$, $i \in \{1, \dots, u_{n-r-1}-u_{n-r}-2\}$ and (i)--(iii) hold true for $(\pi_r^{i-1} \dots \pi_0^0)(\hathat{\eta}_{f,T})$, they also hold true for $(\pi_r^i \dots \pi_0^0)(\hathat{\eta}_{f,T})$.
		\item[(c)] (iv) holds true for $(\pi_r^i \dots \pi_0^0)(\hathat{\eta}_{f,T})$ with $r$ minimal such that $\pi_r \neq \text{id}$ and with $i \in \{0, \dots, u_{n-r-1}-u_{n-r}-2\}$. 
		\item[(d)] If $r \geq 1$, $\pi_r \neq \text{id}$, $(\pi_{r-1} \dots \pi_0) \neq \text{id}$ and (i)--(iv) hold true for $(\pi_{r-1} \dots \pi_0)(\hathat{\eta}_{f,T})$, then (i)--(iii) hold true for $(\pi_r^0 \dots \pi_0^0)(\hathat{\eta}_{f,T})$. 
		\item[(e)] If $r \geq 1$, $\pi_r \neq \text{id}$, $(\pi_{r-1} \dots \pi_0) \neq \text{id}$ and (i)--(iv) hold true for $(\pi_{r-1} \dots \pi_0)(\hathat{\eta}_{f,T})$, then (iv) holds true for $(\pi_r^i \dots \pi_0^0)(\hathat{\eta}_{f,T})$ with $i \in \{0, \dots, u_{n-r-1}-u_{n-r}-2\}$. 
	\end{itemize}
	
	By iteration, (a)--(e) imply Claim 2. 
	
	\paragraph{(a):} For $\pi_r^0$ with $r$ minimal such that $\pi_r\neq \text{id}$, it holds $u_{n-r}-u_n=r$. Thus, (i)--(iii) from Claim 2 become: if $\ell < n$, in $\pi_r^0(\hathat{\eta}_{f,T})$, there are
	\begin{itemize}
		\item[(i')] $\ell+1$ holes at sites $\leq u_n+n-1$; 
		\item[(ii')] at least $n-j-\ell$ holes in $\{u_n+n, \dots, u_{1+j}\}$ for $j=1, \dots, j^*$; 
		\item[(iii')] at least $n-\ell-1$ particles in $\{u_{n-r}+2, \dots, u_n+n-1\}$. 
	\end{itemize}
	Further, recall that the original permutation $\pi_r^0$ exchanges $u_{n-r}+1$ and $u_1$. 
	
	By $A$ and since $\ell < n$, in $\hathat{\eta}_{f,T}$, there is at least one hole in  $\{u_n+n, \dots, u_1\}$ and there are at least $n-\ell \geq 1$ particles in $\{u_{n-r}+1, \dots, u_n+n-1\}$. By Claim 1, $\pi_r^0$ increases the number of holes at sites $\leq u_n+n-1$ by $1$, moves a hole to $u_{n-r}+1$ and a particle to $u_1$. This implies (i') and (iii'). 
	
	Let $j \in \{1, \dots, j^*\}$. If in $\hathat{\eta}_{f,T}$, there was at least one hole in $\{u_{1+j}+1, \dots, u_1\}$, then in $\pi_r^0(\hathat{\eta}_{f,T})$, the number of holes in $\{u_n+n, \dots, u_{1+j} \}$ remains unchanged, that is, at least $n-j-\ell$. Else, by $A$, there were at least $n-\ell$ holes in $\{u_n+n, \dots, u_{1+j} \}$. By $\pi_r^0$, we remove one hole, meaning that there are still at least $n-\ell-1 \geq n-j-\ell$ holes left in that region. This shows (ii'). 
	
	\paragraph{(b):} We obtain (b) by the same arguments as (a), now using that (i)--(iii) hold true for $(\pi_r^{i-1} \dots \pi_0^0)(\hathat{\eta}_{f,T})$.
	
	\paragraph{(c):} 
	Let $j \in \{r+1, \dots, k \}$ and let $z$ be the number of particles in $\hathat{\eta}_{f,T}$ at the sites $\{u_{n-r}+1, \dots, u_{n-j} \}$. 
	As $\ell < n-u_{n-r}+u_n+r-i$, we know by Claim 1 and (ii),(iii) that each permutation $\pi_r^0, \dots \pi_r^i$ increases the number of holes at sites $\leq u_n+n-1$ by $1$. Further, holes are moved to the sites $u_{n-r}+1, \dots, u_{n-r}+1+i$. This means that there is no particle initially at a site in $\{u_{n-r}+1, \dots, u_{n-j} \}$ that is at or is moved to a site where the subsequent permutations of $\pi_r^0, \dots, \pi_r^i$ in $\pi_r$ cannot reach it any more. In particular, $(\pi_r^i \dots \pi_r^0)$ remove $\text{min}(z,i+1)$ particles from $\{u_{n-r}+1, \dots, u_{n-j} \}$. If $z \geq i+1$, then the number of particles in $\{u_{n-j}+1, \dots, u_n+n-1\}$ remains the same. By $A$, it is at least $u_n - u_{n-j} + n - \ell+j$. On the other hand, if $z < i+1$, then in $\hathat{\eta}_{f,T}$ there were at least $u_n- u_{n-r}+n-\ell+r-z$ particles in $\{u_{n-j}+1, \dots, u_n+n-1\}$. Since $i+1-z$ of them are removed, in $(\pi_r^i \dots \pi_0^0)(\hathat{\eta}_{f,T})$ at least $u_n- u_{n-r}+n-\ell+r-i-1$ of them remain. But $i \leq u_{n-r-1}-u_{n-r}-2$ and $u_{n-j} \geq u_{n-r-1}+j-r-1$ imply $u_n- u_{n-r}+n-\ell+r-i-1 \geq u_n- u_{n-j} +n-\ell+j$. This shows (iv) for $(\pi_r^i \dots \pi_0^0)(\hathat{\eta}_{f,T})$.
	
	\paragraph{(d):} We obtain (d) by similar means as (a) and (b), only that in the input, (iii) is replaced by (iv). 
	
	Notice that $u_{n-r}-u_n-r = u_{n-r+1}-u_n-(r-1)+(u_{n-r}-u_{n-r+1}-2)+1$. Given $\ell < n-u_{n-r}+u_n+r$, in $(\pi_{r-1} \dots \pi_0)(\hathat{\eta}_{f,T})$ there are $\ell-u_n+u_{n-r}-r$ holes at sites $\leq u_n+n-1$, at least $n-u_{n-r}+u_n+r-\ell \geq 1$ holes in $\{u_n+n, \dots, u_{1+u_{n-r}-u_n-r} \}$ and at least $u_n-u_{n-r}+n-\ell+r \geq 1$ particles in $\{u_{n-r}+1, \dots, u_n+n-1\}$. Thus, $\pi_r^0$ increases the number of holes at sites $\leq u_n+n-1$ by $1$ and moves a hole to $u_{n-r}+1$. In $\{u_{n-r}+2, \dots, u_n+n-1\}$, there are still at least $u_n-u_{n-r}+n-\ell+r-1$ particles. This shows (i) and (iii). 
	
	Let $j \in \{u_{n-r}-u_n-r+1, \dots, j^*\}$. If in $(\pi_{r-1} \dots \pi_0)(\hathat{\eta}_{f,T})$ there is a hole in $\{ u_{1+j}+1, \dots, u_{1+u_{n-r}-u_n-r} \}$, then $\pi_r^0$ removes one from these sites and leaves the number of holes in $\{u_n+n, \dots, u_{1+j} \}$ unchanged, that is, at least $n-j-\ell$. Else, there were at least $n-u_{n-r}+u_n+r-\ell \geq n-j-\ell+1$ holes at these sites and $\pi_r^0$ removes only one of them. This implies (ii). 
	
	\paragraph{(e):} Having (b) and (d), (e) can be verified by similar means as (c). 
	
	\ \\ Below, we argue that except for (iv), Claim 2 also holds true for $r=k$ if $\pi_k \neq \text{id}$. In doing so, we need a case distinction to determine the last non-trivial permutation. Afterwards, we specify the maximal number of holes added to sites $\leq u_n+n-1$. 
	
	\paragraph{Case 1: $u_n+n-1 = u_{n-k-1}$.}  Then, we have $\pi_k^{u_n+n-1-u_{n-k}-1} = \text{id}$ and it holds
	$\pi_k = \pi_k^{u_{n-k-1}-u_{n-k}-2} \dots \pi_k^0$. We observe
	\begin{equation}
	\begin{aligned}
	& u_{n-k}+(u_{n-k-1}-u_{n-k}-2)+1 = u_n+n-2 < u_n+n-1, \\ 
	& u_{1+u_{n-k}-u_n-k+(u_{n-k-1}-u_{n-k}-2)} = u_{n-k-2} > u_n+n-1.
	\end{aligned} 
	\end{equation} 
	For $i \leq u_{n-k-1}-u_{n-k}-3$, it holds 
	\[u_{n-k}+i+2 \leq u_n+n-1 \ \text{ and } \ u_{n-k}-u_n-k+i+1 \leq n-k-3 = j^*.\] 
	Therefore, for this choice of $i$ and $\ell <n-u_{n-k}+u_n+k-i$, (i)--(iii) apply for $(\pi_k^i \dots \pi_0^0)(\hathat{\eta}_{f,T})$. For $i = u_{n-k-1}-u_{n-k}-2$, (i) holds true, meaning that $\pi_k^{u_{n-k-1}-u_{n-k}-2}$ still increases the number of holes at sites $\leq u_n+n-1$ if $\ell <n-u_{n-k}+u_n+k-i$.
	
	\paragraph{Case 2: $u_n+n-1 < u_{n-k-1}$.} Then, we have $\pi_k = \pi_k^{u_n+n-1-u_{n-k}-1} \dots \pi_k^0$. Similarly as above, we observe that for $i \leq u_n+n-u_{n-k}-3$ and $\ell <n-u_{n-k}+u_n+k-i$, (i)--(iii) apply for $(\pi_k^i \dots \pi_0^0)(\hathat{\eta}_{f,T})$. For $i = u_n+n-1-u_{n-k}-1$, (i) holds true, meaning that the number of holes at sites $\leq u_n+n-1$ is still increased if $\ell < n - u_{n-k}+u_n+k-i$. \\ 
	
	We have seen that as long as $\ell < n - u_{n-r}+u_n+r-i$, the application of $\pi_r^i \ (\neq \text{id})$ increases the number of holes at sites $\leq u_n+n-1$ by $1$. We claim that there are enough non-trivial permutations such that the number eventually reaches $n$. 
	
	If $u_n+n-1 = u_{n-k-1}$, we find 
	$\sum\nolimits_{r=0}^k (u_{n-r-1}-u_{n-r}-1) = n-k-2$ such permutations. By $A$, it holds $\ell \geq k+2$. 
	
	If $u_n+n-1 < u_{n-k-1}$, the number of non-trivial permutations is given by $ \sum_{r=0}^{k-1} (u_{n-r-1}-u_{n-r}-1) + (u_n+n-1-u_{n-k}) = n-k-1$
	and $A$ implies $\ell \geq k+1$. 
	
	Thus, in both cases, in $\pi (\hathat{\eta}_{f,T}) = \hat{\eta}_{f,T}$ there are at least $n$ holes at sites $\leq u_n+n-1$. This means $A \subseteq E$.  
	
	\paragraph{Step 2: $E \subseteq A$.} We show $A^c \subseteq E^c$. It holds 
	\begin{equation}
	A^c = \bigcup_{j=0}^{n-1} \{ \exists \text{ at most } n-j-1 \text{ numbers } i \leq u_{1+j}: \ \hathat{\eta}_{f,T}(i) > s \}.
	\end{equation}
	First, suppose one of the events occurs for $j \in \{0, \dots, n-k-2 \}$. Since $\pi_r^i$ permutes $\{u_{n-r}+1+i, \dots, u_{1+u_{n-r}-u_n-r+i}\}$, the number of permutations that can add a hole to the sites $\leq u_{1+j}$ is $j$. 
	Thus, in $\pi(\hathat{\eta}_{f,T})$, we find at most $n-j-1+j = n-1$ holes at sites $\leq u_{1+j}$. Since $u_{1+j} \geq u_n+n-1$, this means that $E^c$ occurs. 
	
	Secondly, let $j \in \{n-k-1, \dots, n-1 \}$. 
	
	If $u_{1+j} = u_n+n-j-1$, then the permutations do not add any holes to sites $\leq u_{1+j}$. Else, we find $r$ and $i$ maximal such that $u_{n-r}+1+i < u_{1+j}$. Notice that the definition of our permutations does not allow equality. We obtain $ r \leq n-j-2$ and $i = u_{n-r-1}-u_{n-r}-2$. In particular, $(\pi_r^i \dots \pi_0^0)$ can add at most \[u_{n-r}-u_n-r+i+1 \leq u_{1+j}-u_n-n+j+1\] holes to the sites $\leq u_{1+j}$.
	
	Hence, in $\pi(\hathat{\eta}_{f,T})$, we have at most \[(n-j-1)+(u_{1+j}-u_n-n+j+1) = u_{1+j}-u_n\] holes at sites $\leq u_{1+j}$. Since 
	$|\{u_{1+j}+1, \dots, u_n+n-1 \}| = u_n+n-1-u_{1+j}$, this means there are at most $n-1$ holes at sites $\leq u_n+n-1$. Thus, $E^c$ occurs. 
	
	This concludes the proof of $E=A$. 
\end{proof}

\begin{lem} \label{lem:permutation_increases_number_of_holes} 
	Let $\pi$ denote the swap operators corresponding to a permutation of $\mathbb{Z}$ that exchanges $a$ and $b$ for some $a < b \in \mathbb{Z}$ and maps all other integers to themselves. Let $\eta$ be a configuration of a single-species TASEP on $\mathbb{Z}$, and let $x \in \{a, \dots, b-1\}$. Suppose that in $\eta$, there are a hole at a site in $\{x+1, \dots, b \}$ and a particle at a site in $\{a, \dots, x \}$. Then, in $\pi(\eta)$, 
	\begin{itemize}
		\item[(i)] the number of holes at sites in $\{a, \dots, x \}$ increases by $1$; 
		\item[(ii)] a hole is moved to site $a$; 
		\item[(iii)] a particle is moved to site $b$.
	\end{itemize}
\end{lem}
\begin{proof}
	We can write $\pi = \pi^b \pi^a$ with $\pi^b, \pi^a$ denoting the swap operators corresponding to the transpositions 
	$ (b-1,b) \dots (a+1,a+2) $ respectively $ (a,a+1) \dots (b-1,b)$.
	The permutations are applied from right to left. Notice that $(a,a+1) \dots (b-1,b)$ moves the value $b$ to position $a$ and shifts all values $a, \dots, b-1$ to the right by one step. Applied afterwards, $(b-1,b) \dots (a+1,a+2)$ moves the value $a$ to position $b$ and thereby shifts the values $a+1, \dots, b-1$ back to their original positions. 
	
	By the definitions of swap operators and the representations $\pi^a,\pi^b$, we immediately obtain (ii) and (iii). It remains to prove (i).  
	
	First, suppose that in $\eta$, there is a hole at position $x$. Then, $\pi^a$ moves a hole to site $x+1$ and does not change the number of holes in $\{a, \dots, x \}$. Still, a hole is moved to site $a$. For this reason, $\pi^b$ moves a particle to site $x$. As there is a hole at site $x+1$, they exchange their positions. Consequently, the number of holes in $\{a, \dots, x \}$ increases by $1$. 
	
	Secondly, suppose that in $\eta$, there is a particle at position $x$. Then, $\pi^a$ moves a hole from a site $\geq x+1$ to a site $\leq x$ and thereby increases the number of holes in $\{a, \dots, x \}$ by $1$. The particle previously at site $x$ is now at site $x+1$. As a consequence, $\pi^b$ does not change the number of holes in $\{a, \dots, x \}$ any more. 
\end{proof} 

\section{Proofs of the large-time asymptotics} \label{sect:proofs_large_time} 

This section contains the proofs of Theorem~\ref{th:2.1}, Theorem~\ref{th:2.2}, and Theorem~\ref{th:2.3}. 

\subsection{Particles in the interior of the region affected by the wall} \label{sect:proof_asymptotics_interior} 

Before showing Theorem~\ref{th:2.1}, we derive the variational formula \eqref{eq:remark_1} independently of Theorem~\ref{th:finite_time_identity_non_random_IC}.

We denote by $(x^{\textup{step},Z}(t),t\geq 0)$ a TASEP with a shifted step initial condition whose rightmost particle starts at $Z \in \mathbb{Z}$. By Lemma~2.1 of \cite{Sep98c}, we have
\begin{equation} \label{eq:lemma_2.1_sep98} 
\tilde{x}_n(T) = \min_{j \leq n-1} \{x^{\text{step},x^f_{1+j}}_{n-j}(T) \}
\end{equation}  
almost surely, with all processes being coupled by basic coupling. The formula \eqref{eq:lemma_2.1_sep98} holds true for any initial condition $\{x_n^f, n \in \mathbb{Z} \}$.

\begin{lem} \label{lem:colour_position_symmetry_step} 
	Let $I \subseteq \mathbb{N}$, $\{Z_n\}_{n \in I} \subseteq \mathbb{Z}$ and let $(x^{\textup{step},Z_n}(t),t\geq 0), n \in I,$ be coupled by basic coupling. Then, for any $s \in \mathbb{R}$, it holds
	\begin{equation}
	\mathbb{P} \left( \min_{n \in I} \{ x^{\textup{step},Z_n}_n(T) \} \leq s \right) = \mathbb{P} \left( \min_{n \in I} \{ x^{\textup{step},0}_n(T)+Z_n \} \leq s \right). 
	\end{equation}
\end{lem} 

Lemma~\ref{lem:colour_position_symmetry_step} remains valid for random $\{Z_n\}_{n \in I}$ because the time evolution of the TASEPs is constructed using Poisson processes that are independent of $\{Z_n\}_{n \in I}$. Together, Lemma~\ref{lem:colour_position_symmetry_step} and \eqref{eq:lemma_2.1_sep98} imply \eqref{eq:remark_1}.

\begin{proof}[Proof of Lemma~\ref{lem:colour_position_symmetry_step}]
	Lemma~\ref{lem:colour_position_symmetry_step} is shown by colour-position symmetry as follows.
	
	For $s \in \mathbb{R} \setminus \mathbb{Z}$, we can replace $s$ by $\lfloor s \rfloor$. Since $(x^{\textup{step},Z_n}(t),t \geq 0), n \in I,$ are coupled by basic coupling, we can view all these processes as marginals of the same multi-species TASEP $(\eta_t,t \geq 0)$ with packed initial condition. In doing so, we view colours $\leq Z_n$ as particles and colours $> Z_n$ as holes. Then, colour-position symmetry yields 
	\begin{equation}
	\begin{aligned} 
	&\mathbb{P} \left( \min_{n \in I} \{ x^{\textup{step},Z_n}_n(T) \} \leq s \right)\\
	&= \mathbb{P} \left( \bigcup\nolimits_{n \in I} \{ \exists \text{ at most } n-1 \text{ numbers } i \leq Z_n: \ \text{inv}(\eta_T)(i) > s \} \right)\\
	&= \mathbb{P} \left( \bigcup\nolimits_{n \in I} \{ \exists \text{ at most } n-1 \text{ numbers } i \leq Z_n: \ \hat{\eta}_T(i) > s \} \right).
	\end{aligned} 
	\end{equation}
	Here, $(\hat{\eta}_{t},t \in [0,T])$ is another multi-species TASEP with packed initial condition. We couple it with a single-species TASEP $(x^{\textup{step},0}(t),t\geq 0)$ with step initial condition by viewing colours $>s$ as holes, applying the particle-hole-duality and the change of coordinates $z \mapsto s+1-z$. Then, the probability above equals 
	\begin{equation}
	\begin{aligned}
	&\mathbb{P} \left( \bigcup\nolimits_{n \in I} \{ \exists \text{ at most } n-1 \text{ particles at sites } > s-Z_n \} \right) \\ &= \mathbb{P} \left( \min_{n \in I} \{ x^{\textup{step},0}_n(T)+Z_n \} \leq s \right).
	\end{aligned} 
	\end{equation}
	This concludes the proof.
\end{proof}

\begin{proof}[Proof of Theorem~\ref{th:2.1}]
	Theorem~\ref{th:2.1} follows from Theorem~\ref{th:finite_time_identity_non_random_IC} if 
	\begin{equation} \label{eq:pf_thm_2.1_0}
	\lim_{T \to \infty} \mathbb{P} ( x_{\alpha T-j}(T) > \xi T - S T^{1/3} - x_{1+j}^f \ \forall j \in [\alpha T-1]) = 1.
	\end{equation}
	By Lemma~\ref{lem:colour_position_symmetry_step} and by \eqref{eq:lemma_2.1_sep98}, the left hand side above becomes
	\begin{equation}
	\lim_{T \to \infty} \mathbb{P} \left( \min_{j \in [\alpha T-1]} \{ x^{\text{step}, x_{1+j}^f}_{\alpha T-j}(T) \} > \xi T - ST^{1/3} \right) = \lim_{T \to \infty} \mathbb{P} (\tilde{x}_{\alpha T}(T) > \xi T - S T^{1/3}).  
	\end{equation}
	Since $\tilde{x}_{\alpha T}(T) \simeq g_\alpha T$ with $g_\alpha > \xi$, the limit equals $1$. This yields \eqref{eq:pf_thm_2.1_0}. 
\end{proof}

One could also prove \eqref{eq:pf_thm_2.1_0} using one-point estimates for TASEP with step initial condition, see Lemma~\ref{Lemma_A.2}.

\subsection{Particles on the boundary of the region affected by the wall --- Decoupling}  \label{sect:proof_asymptotics_decoupling} 

\begin{proof}[Proof of Theorem~\ref{th:2.2}] 
	The proof of Theorem~\ref{th:2.2} is two-fold: first, we show that there exists some $\delta > 0$ such that 
	\begin{equation} \label{eq:pf_thm_2.2_0} 
	\lim_{T \to \infty} \mathbb{P} (x_{\alpha T-j}(T) > \xi T - S T^{1/3} - x_{1+j}^f \ \forall j \in [\delta T]) = 1. 
	\end{equation}
	Afterwards, we argue that the fluctuations of the particle positions $(x_{\alpha T}(t), t \in [0,T])$ and $(x_{\alpha T-j}(T), j \in [\alpha T-1] \setminus [\delta T])$ are asymptotically independent. By Theorem~\ref{th:finite_time_identity_non_random_IC}, these facts imply
	\begin{equation}
	\begin{aligned}
	& \lim_{T \to \infty} \mathbb{P}(x^f_{\alpha T}(T) > \xi T - S T^{1/3}) \\
	&= \lim_{T \to \infty} \mathbb{P}(x_{\alpha T}(t) > \xi T - f(T-t)-ST^{1/3} \ \forall t \in [0,T]) \\
	&\hphantom{= } \times \lim_{T \to \infty} \mathbb{P}(x_{\alpha T-j}(T) > \xi T - S T^{1/3} - x_{1+j}^f \ \forall j \in [\alpha T-1] \setminus [\delta T]). 
	\end{aligned}  
	\end{equation}
	By \eqref{eq:lemma_2.1_sep98} and \eqref{eq:pf_thm_2.2_0}, the second limit equals $\lim_{T \to \infty} \mathbb{P} ( \tilde{x}_{\alpha T}(T) > \xi T - S T^{1/3})$. This shows Theorem~\ref{th:2.2}.
	
	\paragraph{Proof of \eqref{eq:pf_thm_2.2_0}.}  We set $\eps = \tfrac{1}{2}((1-2\sqrt{\alpha})-\xi) > 0$. By \eqref{eq:linear_decay_IC}, we find 
	\begin{equation}
	\xi T - x_{1+j}^f \leq (1-2\sqrt{\alpha})T - 2 \eps T + d j \leq (1-2\sqrt{\alpha-jT^{-1}})T + (d-\tfrac{1}{\sqrt{\alpha}})j - 2 \eps T,
	\end{equation}
	up to $\mathcal{O}(1)$. We choose $d > \tfrac{1}{\sqrt{\alpha}}$ and $\delta < \eps (d-\tfrac{1}{\sqrt{\alpha}})^{-1}$ with $\delta < \alpha$. Then, it holds
	\begin{equation} \label{eq:pf_thm_2.2_1} 
	\begin{aligned}
	&\mathbb{P} ( x_{\alpha T-j}(T) \leq \xi T - S T^{1/3} - x_{1+j}^f \text{ for some } j \in [\delta T]) \\ 
	&\leq \mathbb{P} ( x_{\alpha T-j}(T) \leq (1-2\sqrt{\alpha-jT^{-1}})T - \eps T \text{ for some } j \in [\delta T]) \\
	&\leq \sum\nolimits_{j \in [\delta T]}  \mathbb{P} ( x_{\alpha T-j}(T) \leq (1-2\sqrt{\alpha-jT^{-1}})T - \eps T ) \\
	&\leq C T e^{-c T^{2/3}}.
	\end{aligned} 
	\end{equation}
	In the last step, we apply one-point estimates, see Lemma~\ref{Lemma_A.2}. The constants are uniform because $\alpha-jT^{-1} \in [\alpha-\delta,\alpha] \subseteq (0,1)$. In particular, \eqref{eq:pf_thm_2.2_1} implies \eqref{eq:pf_thm_2.2_0}. 
	
	\paragraph{Proof of the asymptotic independence.} Next, we argue that the fluctuations of $(x_{\alpha T}(t), t \in [0,T])$ and $(x_{\alpha T-j}(T), j \in [\alpha T-1]\setminus [\delta T])$ are asymptotically independent. In doing so, we only sketch the arguments as their rigorous implementation is similar to Section~6 of \cite{FG24}. It suffices to consider $x_{(\alpha-\delta)T}(T)$ instead of $(x_{\alpha T-j}(T), j \in [\alpha T-1]\setminus [\delta T])$ because if asymptotic independence of the tagged particle process holds true for the leftmost particle, then it holds true for all of them. 
	
	As seen in \cite{FG24}, we can show asymptotic independence of a tagged particle of events in a deterministic space-time region by localising its backwards path outside of this region with probability converging to $1$. 
	The rightmost backwards path of those starting at $(x_{\alpha T}(t),t \in [0,T])$ is the one of $x_{\alpha T}(T)$. By Proposition~4.2 of \cite{FG24}, for any small $\iota > 0$, it stays to the left of the line $\ell_1(t) = (1-2\sqrt{\alpha})t+T^{2/3+\iota}$ with probability converging to $1$. Therefore, Lemma~3.1 of \cite{FG24} implies that $(x_{\alpha T}(t), t \in [0,T])$ are asymptotically independent of all events to the right of the line $\ell_1(t), t \in [0,T]$.
	
	On the other hand, slow decorrelation, for example Corollary~5.2 of \cite{FG24} at a fixed time, states that the asymptotic fluctuations of $x_{(\alpha-\delta)T}(T)$ are the same as those of $x^{\text{step},x_{(\alpha-\delta)T^\nu}(T^\nu)}_{(\alpha-\delta)(T-T^\nu)}(T^\nu,T)$ in a TASEP with a shifted step initial condition started at time $T^\nu$, for any $\nu \in (0,1)$. The processes are coupled by basic coupling. We choose $\nu \in (\tfrac{2}{3}+\iota,1)$. Given a localisation of $x_{(\alpha-\delta)T^\nu}(T^\nu)$ that will hold with probability converging to $1$, it can be shown as in \cite{FG24} that the asymptotic fluctuations of $x^{\text{step},x_{(\alpha-\delta)T^\nu}(T^\nu)}_{(\alpha-\delta)(T-T^\nu)}(T^\nu,T)$ are independent of all events during $[0,T^\nu]$ and of all events to the left of the line $\ell_2(t) = (1-2\sqrt{\alpha-\delta})t - T^{2/3+\iota}$ during $[T^\nu,T]$. Since for $T$ large enough, we have $\ell_1(t) < \ell_2(t)$ for all $t \in [T^\nu,T]$, this yields the asymptotic independence we were looking for. 
\end{proof}

\subsection{Particles on the boundary of the region affected by the wall --- Interpolation}   \label{sect:proof_asymptotics_interpolation} 

\begin{proof}[Proof of Theorem~\ref{th:2.3}]
	We prove Theorem~\ref{th:2.3} for $\alpha_0 = 1$ in Assumption~\ref{assumpt:2.1.1}. If there are several macroscopic time regions of wall influences, the limit distribution factorizes as proven in \cite{FG24}. Indeed, as the fluctuations of $x_{\alpha T}(t)$ for $t \in [\alpha_i T - \varkappa T^{2/3}, \alpha_i T + \varkappa T^{2/3} ]$ are asymptotically independent of $x_{\alpha T}(t)$ at different macroscopic times $t$, for $\alpha_i < 1$ they are also asymptotically independent of $(x_{\alpha T-j}(T), j \in [\alpha T-1])$. 
	
	In the following, we write $c_1 = c_1^0$, $c_2 = c_2^0$, $g_T = g_T^0$ and $g = g_0$. In particular, we recall $c_1 = (1-\sqrt{\alpha})^{2/3} \alpha^{-1/6}$, $c_2 = 2(1-\sqrt{\alpha})^{1/3}\alpha^{-1/3}$ and $\hat{c}_2  =\alpha c_2$.
	
	We first suppose that $\{x_n^f, n \in \mathbb{N}\}$ fulfils Assumption~\ref{assumpt:2.7} with $d < \tfrac{1}{\sqrt{\alpha}}$. Let $\eps > 0$ be arbitrarily small but fixed. Recalling Theorem~\ref{th:finite_time_identity_non_random_IC}, our first observation is
	\begin{equation} \label{eq:pf_thm_2.3_0}
	\begin{aligned}
	& \lim_{T \to \infty} \mathbb{P}(x^f_{\alpha T}(T) > (1-2\sqrt{\alpha})T-ST^{1/3}) \\
	&= \lim_{T \to \infty} \mathbb{P} (x_{\alpha T}(t) > (1-2\sqrt{\alpha})T - f(T-t) - S T^{1/3} \ \forall t \in [0,T], \\ 
	& \hphantom{=\lim_{T \to \infty} \mathbb{P} (} \ x_{\alpha T-j}(T) > (1-2\sqrt{\alpha})T - S T^{1/3} -x_{1+j}^f \ \forall j \in [T^{1/3+\eps}]).
	\end{aligned}
	\end{equation}
	Indeed, it holds 
	\begin{equation} \label{eq:pf_thm_2.3_0.3} 
	\begin{aligned} 
	x_{1+j}^f =  - dj + c_1 y_T(\hat{c}_2^{-1} j T^{-2/3}) T^{1/3} 
	\geq & - d j - \tfrac{1}{8 \alpha^{3/2}} j^2 T^{-1}  \\
	\geq & - \tfrac{1}{\sqrt{\alpha}} j + 2 \delta j - \tfrac{1}{8 \alpha^{3/2}} j^2 T^{-1} 
	\end{aligned} 
	\end{equation} 
	for $\delta = \tfrac{1}{2}(\tfrac{1}{\sqrt{\alpha}}-d) > 0$, and
	\begin{equation} \label{eq:pf_thm_2.3_0.6} 
	(1-2\sqrt{\alpha-j T^{-1}}) T \geq (1-2\sqrt{\alpha}) T + \tfrac{1}{\sqrt{\alpha}} j  + \tfrac{1}{4 \alpha^{3/2}} j^2 T^{-1} .
	\end{equation}
	For $\eta > 0$ small, this implies
	\begin{equation}
	\begin{aligned}
	& \mathbb{P} (\exists j \in [(\alpha-\eta)T] \setminus [T^{1/3+\eps}]: \ x_{\alpha T-j}(T) \leq (1-2\sqrt{\alpha})T - S T^{1/3} -x_{1+j}^f  ) \\
	& \leq \mathbb{P} (\exists j \in [(\alpha-\eta)T] \setminus [T^{1/3+\eps}]: \ x_{\alpha T-j}(T) \leq (1-2\sqrt{\alpha-jT^{-1}})T - \delta j  ) ,
	\end{aligned} \end{equation} 
	which we bound by
	\begin{equation}
	\sum_{j \in [(\alpha-\eta)T] \setminus [T^{1/3+\eps}]} \mathbb{P} (x_{\alpha T-j}(T) \leq (1-2\sqrt{\alpha-jT^{-1}})T - \delta T^{1/3+\eps} )  \leq C T e^{-c T^{\eps}}.
	\end{equation}
	In the last step, we apply one-point estimates and get uniform constants because $\alpha - j T^{-1} \in [\eta, \alpha] \subseteq (0,1)$. Further, we find
	\begin{equation}
	\begin{aligned}
	& \mathbb{P} (\exists j \in [\alpha T-1] \setminus [(\alpha-\eta)T]: \ x_{\alpha T-j}(T) \leq (1-2\sqrt{\alpha})T - S T^{1/3} -x_{1+j}^f  ) \\
	&\leq \mathbb{P}(x_{\eta T}(T) \leq (1-2\sqrt{\alpha})T-ST^{1/3}-x_{\alpha T}^f).
	\end{aligned} 
	\end{equation}
	Choose $\eta > 0$ such that $(1-2\sqrt{\eta}) > (1-2\sqrt{\alpha})-x_{\alpha T}^fT^{-1}$, then by one-point estimates, the probability is $\leq C e^{-c T^{2/3}}$. We deduce \eqref{eq:pf_thm_2.3_0}.
	
	Thus, adapting Lemma~4.11 -- Lemma~4.14 of \cite{BBF21}, we need to determine the asymptotic probability of the event 
	\begin{equation} \label{eq:pf_thm_2.3_1}
	\begin{aligned} 
	\{  &x_{\alpha T}(T-c_2\tau T^{2/3}) > (1-2\sqrt{\alpha})T - (1-\sqrt{\alpha})c_2 \tau T^{2/3} + c_1(\tau^2-g_T(\tau))T^{1/3} \\ &  - S T^{1/3} \ \forall \tau \in [0,\varkappa] \} \cap \{ x_{\alpha T-j}(T) > (1-2\sqrt{\alpha})T - S T^{1/3} -x_{1+j}^f \ \forall j \in [T^{1/3+\eps}] \},
	\end{aligned} 
	\end{equation}
	where we first take $T \to \infty$ and then $\varkappa \to \infty$. More precisely, 
	\begin{equation} \label{eq:pf_thm_2.3_1.5}
	\begin{aligned}
	& \lim_{T \to \infty} \mathbb{P}(x^f_{\alpha T}(T) > (1-2\sqrt{\alpha})T-ST^{1/3}) = \lim_{\varkappa \to \infty} \lim_{T \to \infty} \mathbb{P} ( \eqref{eq:pf_thm_2.3_1} ). 
	\end{aligned} 
	\end{equation}
	By Theorem~4.7 of \cite{BBF21}, the limit is bounded from above by 
	\begin{equation}
	\mathbb{P} \Bigl ( \sup_{\tau \geq 0} \{ \mathcal{A}_2(\tau)-g(\tau) \} < c_1^{-1} S \Bigr ).
	\end{equation}
	We want to rewrite \eqref{eq:pf_thm_2.3_1} in terms of the process along a space-like path, such that the limit distribution can be determined as in \cite{BF07}. In doing so, we consider coordinates $\omega^1 = \tfrac{1}{2}(t-n)$ and $\omega^0 = \tfrac{1}{2}(t+n)$, where the variables $t$ and $n$ correspond to times respectively labels. We define 
	\begin{equation}
	\pi(\omega^1) = \begin{cases}  \alpha + \omega^1 \ & \text{ if } \  \omega^1 \in [-\tfrac{\alpha}{2}, \tfrac{1-\alpha}{2}], \\
	1-\omega^1 \ & \text{ if } \  \omega^1 \in (\tfrac{1-\alpha}{2},\tfrac{1}{2}], \end{cases} 
	\end{equation}
	see also Figure~\ref{fi:space_like_path}.
	\begin{figure}[t!]
		\centering
		\includegraphics[scale=0.7]{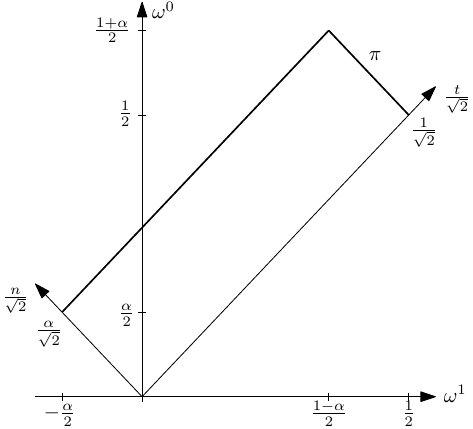}
		\caption{The space-like path $\pi$ in the proof of Theorem~\ref{th:2.3}.}
		\label{fi:space_like_path}
	\end{figure}
	Notice that $\pi$ is \emph{not} a smooth function as required in \cite{BF07}\footnote{In \cite{BF07}, they consider smooth space-like paths and use their first and second derivatives to define the scaling of times and labels around a given point, such that it matches the macroscopic approximation of the particle's position. However, the fact that the parameters are derivatives is not needed in their arguments. Thus, we can find a suitable scaling also for non-smooth paths.}. As the derivative of $\pi$ at $\theta \coloneqq \tfrac{1-\alpha}{2}$ does not exist, we work with its approximations from the left and from the right instead. Following the arguments of \cite{BF07}, for the large-time limit, we consider $\omega^1(u) = \theta T - u T^{2/3}$ with $u \in \mathbb{R}$ and, consequently, 
	\begin{equation}
	\omega^0(u) = \begin{cases}  \tfrac{1+\alpha}{2} T - u T^{2/3} \ & \text{ if } \ u \geq 0, \\ \tfrac{1+\alpha}{2} T + u T^{2/3}  \ & \text{ if } \ u < 0.  \end{cases} 
	\end{equation}
	The scaling for times and labels along the path $\pi$ becomes
	\begin{equation}
	\begin{aligned}
	t(u) = T-2uT^{2/3} , \ n(u) = \alpha T \ & \text{ if } \ u \geq 0, \\ 
	t(u) = T, \ n(u) = \alpha T + 2 u T^{2/3} \ & \text{ if } \ u < 0. 
	\end{aligned} 
	\end{equation}
	Like in (2.19) of \cite{BF07}, we find $x_{n(u)}(t(u)) \simeq \mathbf{x}(u)$ with 
	\begin{equation}
	\mathbf{x}(u) = \begin{cases} (1-2\sqrt{\alpha})T -2(1-\sqrt{\alpha})uT^{2/3}+ \sqrt{\alpha} u^2 T^{1/3} \ & \text{ if } \ u \geq 0, \\ 
	(1-2\sqrt{\alpha})T -2 \tfrac{1}{\sqrt{\alpha}} u T^{2/3} + \tfrac{1}{\alpha^{3/2}} u^2 T^{1/3}  \ & \text{ if } \ u < 0. \end{cases} 
	\end{equation}
	For any $\kappa > 0$, \eqref{eq:pf_thm_2.3_1} contains the event 
	\begin{equation} \label{eq:pf_thm_2.3_2} \begin{aligned} 
	& \{  x_{n(\frac{1}{2}c_2\tau)}(t(\tfrac{1}{2}c_2\tau)) >\mathbf{x}(\tfrac{1}{2}c_2\tau) - c_1g_T(\tau)T^{1/3}  - S T^{1/3} \ \forall \tau \in [0,\varkappa] \} \\ & \cap \{ x_{n(\frac{1}{2}c_2\tau)}(t(\tfrac{1}{2}c_2\tau)) > (1-2\sqrt{\alpha})T - S T^{1/3} -x_{1-c_2 \tau T^{2/3}}^f \ \forall \tau \in [-\kappa,0] \}
	\end{aligned} \end{equation}
	for all $T$ large enough. Recalling \eqref{eq:pf_thm_2.3_1.5}, this means we can bound 
	\begin{equation}
	\lim_{T \to \infty} \mathbb{P}(x^f_{\alpha T}(T) > (1-2\sqrt{\alpha})T-ST^{1/3}) \geq \lim_{\varkappa \to \infty} \lim_{\kappa \to 0} \lim_{T \to \infty} \mathbb{P}(\eqref{eq:pf_thm_2.3_2}). 
	\end{equation}
	By \eqref{eq:pf_thm_2.3_0.3}, we find 
	\begin{equation}
	(1-2\sqrt{\alpha})T - S T^{1/3} - x_{1-c_2 \tau T^{2/3}}^f \leq \mathbf{x}(\tfrac{1}{2}c_2\tau)-\tfrac{1}{2} \tfrac{c_1}{\alpha^2}  \tau^2 T^{1/3}-S T^{1/3}
	\end{equation}
	for $\tau < 0$. Thus, setting 
	\begin{equation}
	\hat{X}_T(u) = \frac{x_{n(u)}(t(u)) - \mathbf{x}(u)}{-T^{1/3}},
	\end{equation}
	\eqref{eq:pf_thm_2.3_2} contains the event 
	\begin{equation} \begin{aligned} 
	&\{ \hat{X}_T(\tfrac{1}{2} c_2 \tau) - c_1 g_T(\tau) < S \ \forall \tau \in [0,\varkappa] \} \cap \{ \hat{X}_T(\tfrac{1}{2} c_2 \tau) - \tfrac{1}{2} \tfrac{c_1}{\alpha^2}\tau^2 < S  \ \forall \tau \in [-\kappa,0]\}.
	\end{aligned} \end{equation}
	As stated in (2.23) of \cite{BF07}, we get 
	\begin{equation} \label{eq:pf_thm_2.3_3} 
	\hat{X}_T(u) \to c_1 \mathcal{A}_2\Bigl(\tfrac{\alpha^{1/3}}{(1-\sqrt{\alpha})^{1/3}} ( \mathbb{1}_{u \geq 0} + \tfrac{1}{\alpha} \mathbb{1}_{u < 0} ) u \Bigr) 
	\end{equation}
	in the sense of finite-dimensional distributions. Indeed, the expression of the finite-dimensional distributions as Fredholm determinants, Proposition~3.1 of \cite{BF07}, holds true for any space-like path. The convergence of the rescaled kernel for step initial condition to the extended Airy kernel $K_{\mathcal{A}_2}$, as sketched in Section~5.2 of \cite{BF07}, transfers to our (non-smooth) choice of $\pi$. We obtain the same representation in terms of functions $f_0, f_1, f_2$ fulfilling (5.55) of \cite{BF07}, only that in our setting,
	\begin{equation}
	\mu = 1-\sqrt{\alpha}, \ \kappa_0 = \tfrac{1}{\sqrt{\alpha}(1-\sqrt{\alpha})}, \ \kappa_1 = \tfrac{1}{1-\sqrt{\alpha}}(\mathbb{1}_{u_i \geq 0} + \tfrac{1}{\alpha} \mathbb{1}_{u_i < 0} ).
	\end{equation}
	Thus, up to conjugation, the rescaled kernel $K^{\text{resc}}(u_1,s_1;u_2,s_2)$ converges to 
	\begin{equation}
	c_1^{-1} K_{\mathcal{A}_2} \Bigl (\tfrac{\alpha^{1/3}}{(1-\sqrt{\alpha})^{1/3}}(\mathbb{1}_{u_1 \geq 0} + \tfrac{1}{\alpha} \mathbb{1}_{u_1 < 0})u_1, c_1^{-1} s_1; \tfrac{\alpha^{1/3}}{(1-\sqrt{\alpha})^{1/3}}(\mathbb{1}_{u_2 \geq 0} + \tfrac{1}{\alpha}\mathbb{1}_{ u_2 < 0})u_2, c_1^{-1} s_2 \Bigr ).
	\end{equation} 
	The approach of \cite{BF07}, which proves convergence by representing a distribution as a Fredholm determinant and showing convergence of the related correlation kernel, is very common and has been used in various contexts. Just to name a few, we refer to \cite{BF07,BFP06,BFS07b,BFS07,FV13,Jo00b,Jo03,Sas05,TW08b}. 
	
	Moreover, by Proposition~2.9 of \cite{BBF21} and Lemma~\ref{lem:weak_conv}, $(\hat{X}_T(u))$ is tight on compact intervals when restricted to $u \geq 0$ or $u \leq 0$. This implies tightness of $(\hat{X}_T(u))$ on compact intervals as well, meaning that \eqref{eq:pf_thm_2.3_3} holds true as weak convergence with respect to the uniform topology on compact sets. 
	Similarly as in the proof of Theorem~2.5 of \cite{FG24}, we obtain 
	\begin{equation} \label{eq:pf_thm_2.3_4} \begin{aligned} 
	& \lim_{T \to \infty} \mathbb{P}( \hat{X}_T(\tfrac{1}{2} c_2 \tau) - c_1 g_T(\tau) < S \ \forall \tau \in [0,\varkappa] , \ \hat{X}_T(\tfrac{1}{2} c_2 \tau) - \tfrac{1}{2}\tfrac{c_1}{\alpha^2}\tau^2 < S \ \forall \tau \in [-\kappa,0] ) \\
	&= \mathbb{P} \Bigl ( \sup_{\tau \in [0,\varkappa]} \{\mathcal{A}_2(\tau) - g(\tau) \} < c_1^{-1} S, \ \sup_{\tau \in [-\kappa,0]} \{\mathcal{A}_2(\tfrac{\tau}{\alpha}) - \tfrac{1}{2}(\tfrac{\tau}{\alpha})^2 \} < c_1^{-1} S \Bigr ).
	\end{aligned} \end{equation}
	Next, we let $\kappa \to 0$ and observe that $\{ \sup_{\tau \in [-\kappa,0]} \{\mathcal{A}_2(\tfrac{\tau}{\alpha}) -  \tfrac{1}{2} (\tfrac{\tau}{\alpha})^2 \} < c_1^{-1} S \}$ is a monotonically increasing sequence of events. Conditioning our probability space on $\mathcal{A}_2$ having continuous sample paths, the sequence converges to $\{\mathcal{A}_2(0) < c_1^{-1} S \}$ as $\kappa \to 0$. Since $f(0) = 0$ implies $g(0)=0$, $\sup_{\tau \in [0,\varkappa]} \{\mathcal{A}_2(\tau) - g(\tau) \} < c_1^{-1} S$ implies $\mathcal{A}_2(0) < c_1^{-1} S$, meaning 
	\begin{equation}
	\lim_{\kappa \to 0} \mathbb{P} \Bigl ( \sup_{\tau \in [-\kappa,0]} \{\mathcal{A}_2(\tfrac{\tau}{\alpha}) - \tfrac{1}{2}(\tfrac{\tau}{\alpha})^2 \} < c_1^{-1} S  \ \Bigl | \  \sup_{\tau \in [0,\varkappa]} \{\mathcal{A}_2(\tau) - g(\tau) \} < c_1^{-1} S \Bigr ) = 1.
	\end{equation}
	Thus, again by similar arguments as in the proof of Theorem~2.5 of \cite{FG24}, our limit probability is bounded from below by 
	\begin{equation}
	\eqref{eq:pf_thm_2.3_1.5} \geq \lim_{\varkappa \to \infty} \lim_{\kappa \to0} \eqref{eq:pf_thm_2.3_4} = \mathbb{P} \Bigl ( \sup_{\tau \geq 0} \{ \mathcal{A}_2(\tau)-g(\tau) \} < c_1^{-1} S \Bigr ).
	\end{equation}
	We conclude 
	\begin{equation}
	\eqref{eq:pf_thm_2.3_1.5} = \mathbb{P} \Bigl ( \sup_{\tau \geq 0} \{ \mathcal{A}_2(\tau)-g(\tau) \} < c_1^{-1} S \Bigr )
	\end{equation}
	for $d < \tfrac{1}{\sqrt{\alpha}}$.
	
	Now, let $d = \tfrac{1}{\sqrt{\alpha}}$. We claim 
	\begin{equation} \label{eq:pf_thm_2.3_5}
	\begin{aligned}
	& \lim_{T \to \infty} \mathbb{P}(x^f_{\alpha T}(T) > (1-2\sqrt{\alpha})T - S T^{1/3}) \\ 
	&= \lim_{\varkappa \to \infty} \lim_{T \to \infty} \mathbb{P}(x_{\alpha T}(t) > (1-2\sqrt{\alpha})T-f(T-t)-ST^{1/3} \ \forall t \in [T-c_2\varkappa T^{2/3},T], \\ 
	& \hphantom{= \lim_{\varkappa \to \infty} \lim_{T \to \infty} \mathbb{P}(} \ x_{\alpha T-j}(T) > (1-2\sqrt{\alpha})T-x_{1+j}^f - S T^{1/3} \ \forall j \in [\hat{c}_2 \varkappa T^{2/3}] ).
	\end{aligned} 
	\end{equation}
	Indeed, \eqref{eq:pf_thm_2.3_5} can be obtained by combining Lemma~4.11 -- Lemma~4.14 of \cite{BBF21} and Corollary~\ref{cor:6.3}. 
	
	By similar arguments as above, the event in the second probability now equals 
	\begin{equation} \label{eq:pf_thm_2.3_6} \{ \hat{X}_T(\tfrac{1}{2}c_2\tau) - c_1 g_T(\tau) < S \ \forall \tau \in [0,\varkappa], \hat{X}_T(\tfrac{1}{2}\hat{c}_2\tau)-c_1\tau^2 - c_1 y_T(- \tau) < S \ \forall \tau \in [-\varkappa,0] \}.
	\end{equation} 
	Since $y_T \to y$ uniformly on compact sets, the weak convergence \eqref{eq:pf_thm_2.3_3} and Lemma~6.6 of \cite{FG24} applied to $\tau \mapsto \tau^2 + y(-\tau)$ yield
	\begin{equation} \begin{aligned} 
	& \lim_{T \to \infty} \mathbb{P} (\eqref{eq:pf_thm_2.3_6}) \\
	&= \mathbb{P} \Bigl ( \sup_{\tau \in [0,\varkappa]} \{ \mathcal{A}_2(\tau)-g(\tau) \} < c_1^{-1} S, \ \sup_{\tau \in [-\varkappa,0]} \{ \mathcal{A}_2(\tau)-\tau^2 - y (- \tau)\} < c_1^{-1} S \Bigr ).
	\end{aligned} \end{equation}
	Since $\tau^2 + y(-\tau) \geq \tfrac{1}{2} \tau^2$, we conclude by similar means as in \cite{FG24}, see Proposition~4.4 of \cite{CH11} and Proposition~2.13(b) of \cite{CLW16}, that \eqref{eq:pf_thm_2.3_5} equals
	\begin{equation}
	\mathbb{P} \Bigl ( \sup_{\tau \geq 0} \{ \mathcal{A}_2(\tau)-g(\tau) \} < c_1^{-1} S, \ \sup_{\tau  \leq 0} \{ \mathcal{A}_2(\tau)-\tau^2 -  y (- \tau) \} < c_1^{-1} S \Bigr ).
	\end{equation}
\end{proof}

\section{Product limit distributions and shocks}  \label{sect:proof_density_shock}

\begin{proof}[Proof of Lemma~\ref{lemma_density_wall_influence_times}] Without loss of generality, suppose $g_i(0) < + \infty$ in Assumption~\ref{assumpt:2.1.1}. Then, it holds
	\begin{equation} \label{eq:pf_lem_density_wall_influence_times_0} 
	f((1-\alpha_i)T) = \xi T - \Bigl ( 1-2\sqrt{\tfrac{\alpha}{\alpha_i}} \Bigr ) \alpha_i T + \mathcal{O}(T^{1/3})
	\end{equation}
	for $i =0, \dots, n$. For any sequence $\eta \downarrow 0$, we have $x^f_{(\alpha-\eta)T}(T) \simeq \xi_\eta T$ only if 
	\begin{equation} \label{eq:pf_lem_density_wall_influence_times_1} 
	\Bigl ( 1 - 2 \sqrt{ \tfrac{\alpha-\eta}{\alpha_i}} \Bigr ) \alpha_i T \geq \xi_\eta T - f((1-\alpha_i)T) + \mathcal{O}(T^{1/3}). 
	\end{equation}
	Expanding the left hand side in $\eta=0$, \eqref{eq:pf_lem_density_wall_influence_times_0} and \eqref{eq:pf_lem_density_wall_influence_times_1} imply 
	$\xi_\eta \leq \xi + \sqrt{\tfrac{\alpha_i}{\alpha}} \eta + \mathcal{O}(\eta^2)$. Therefore, we find
	\begin{equation} \label{eq:pf_lem_density_wall_influence_times_1.5}
	\rho^f_r(\xi) = \liminf_{\eta \downarrow 0} \frac{\eta }{\xi_\eta  - \xi } \geq \liminf_{\eta \downarrow 0} \frac{\eta}{\sqrt{\tfrac{\alpha_i}{\alpha}} \eta + \mathcal{O}(\eta^2)} = \sqrt{\tfrac{\alpha}{\alpha_i}}
	\end{equation}
	for $i=0, \dots, n$.
	
	Next, suppose 
	$ \limsup_{\eta \downarrow 0} \tfrac{\eta}{\xi_\eta-\xi} \geq \sqrt{\alpha \alpha_0^{-1}} + 2 \sigma $
	for some $\sigma > 0$. Then, along some sequence $\eta \downarrow 0$, it holds $ \xi_\eta \leq \xi + (\sqrt{\alpha \alpha_0^{-1}} + \sigma)^{-1} \eta $ for all $\eta$ small enough. We claim that, as a consequence, we see no wall influence on $x^f_{(\alpha-\eta)T}(T)$, which is contradictory. Indeed, since $x^f_{\alpha T}(T) \simeq \xi T$, we obtain 
	\begin{equation}
	\xi_\eta T - f((1-\beta)T) \leq \Bigl ( 1- 2 \sqrt{\tfrac{\alpha-\eta}{\beta}} \Bigr ) \beta T - \sqrt{\tfrac{\beta}{\alpha}} \eta T + \Bigl (\sqrt{\tfrac{\alpha}{ \alpha_0}} + \sigma \Bigr )^{-1} \eta  T + \mathcal{O}(T^{1/3})
	\end{equation}
	for $\beta \in [\alpha,1]$.
	Thus, a wall influence at time $\beta T$ can only happen if $\beta < \alpha_0$.
	Further, Assumption~\ref{assumpt:2.1.1} (a) for $x^f_{\alpha T}(T)$ implies for all $\eta$ small enough:   
	\begin{equation} \begin{aligned} 
	\xi_\eta T - f((1-\beta)T) \leq \begin{cases} -(\alpha-\eta)T-\tfrac{1}{2}K(\eps) T & \text{ if } \beta\in [0,\alpha-\eta], \\  \Bigl ( 1- 2 \sqrt{\tfrac{\alpha-\eta}{\beta}} \Bigr ) \beta T  - \tfrac{1}{2} K(\eps)T & \text{ if } \beta \in [\alpha-\eta,\alpha_0-\eps). \end{cases} 
	\end{aligned} \end{equation} 
	If $\beta = \alpha_0-\delta$ with $\delta \in (0,\eps]$, we know by Assumption~\ref{assumpt:2.1.1} (b) that up to $\mathcal{O}(T^{1/3})$,
	\begin{equation}
	\xi_\eta T - f((1-\beta)T) \leq  \Bigl (1-2\sqrt{\tfrac{\alpha}{\alpha_0}} \Bigr )\alpha_0 T + \Bigl (\sqrt{\tfrac{\alpha}{ \alpha_0}} + \sigma \Bigr )^{-1} \eta  T - \Bigl(1-\sqrt{\tfrac{\alpha}{\alpha_0}}\Bigr ) \delta T + \tfrac{\sqrt{\alpha}}{8 \alpha_0^{3/2}} \delta^2 T. 
	\end{equation}
	Since we have
	\begin{equation}
	\Bigl ( 1- 2 \sqrt{\tfrac{\alpha-\eta}{\beta}} \Bigr ) \beta T  \geq \Bigl (1-2\sqrt{\tfrac{\alpha}{\alpha_0}} \Bigr ) \alpha_0 T - \Bigl (  1 - \sqrt{\tfrac{\alpha}{\alpha_0}} \Bigr ) \delta T + \tfrac{\sqrt{\alpha}}{4 \alpha_0^{3/2}} \delta^2 T + \sqrt{\tfrac{\alpha_0-\delta}{\alpha}} \eta T,
	\end{equation}
	this leads to
	\begin{equation}
	\Bigl ( 1- 2 \sqrt{\tfrac{\alpha-\eta}{\beta}} \Bigr ) \beta T - (\xi_\eta T - f((1-\beta)T)) \geq \tfrac{\sqrt{\alpha}}{8 \alpha_0^{3/2}} \delta^2 T - \tfrac{1}{ \sqrt{\alpha \alpha_0}} \delta \eta T + 2 c \eta T 
	\end{equation}
	for some $c > 0$ depending on $\sigma$. We choose $\eta$ small such that $c \eta \geq \tfrac{2\sqrt{\alpha_0}}{\alpha^{3/2}} \eta^2$. Then, the right hand side is larger than 
	$\tfrac{\sqrt{\alpha}}{8 \alpha_0^{3/2}} (\delta - 4 \tfrac{\alpha_0}{\alpha} \eta)^2 T + c \eta T \geq c \eta T$. 
	We conclude that $\xi_\eta T - f(T-t)$ is macroscopically smaller than the law of large numbers of $x_{(\alpha-\eta)T}(t)$ for all $t \in [0,T]$. This means that the wall has no influence on $x^f_{(\alpha-\eta)T}(T)$, which is contradictory. Thus, we obtain 
	\begin{equation}
	\lim_{\eta \downarrow 0} \frac{\eta }{\xi_\eta  - \xi } = \sqrt{\tfrac{\alpha}{\alpha_0}}.
	\end{equation}
	
	Next, we consider $\rho^f_l(\xi)$ and write $x^f_{(\alpha+\eta)T}(T) \simeq \xi_\eta T$. Analogously to \eqref{eq:pf_lem_density_wall_influence_times_1.5}, we find $\rho^f_l(\xi) \leq \sqrt{\alpha \alpha_n^{-1}}$. 
	We let $x^f_{\alpha T}(T)$ be in the interior of the region affected by the wall, meaning $x^f_{\beta T}(T)$ is also affected for $\beta > \alpha$ close to $\alpha$. Further, we suppose $\liminf_{\eta \downarrow 0}  \tfrac{\eta}{\xi-\xi_\eta} \leq \sqrt{ \tfrac{\alpha}{\alpha_n}} - 2 \sigma$ for some $\sigma > 0$. Then, for all $\eta$ small enough along some sequence $\eta \downarrow 0$, it holds $\xi_\eta \leq \xi - (\sqrt{\tfrac{\alpha}{\alpha_n}}-\sigma)^{-1} \eta$ and $x^f_{(\alpha+\eta)T}(T)$ experiences a wall influence. By similar arguments as before, a wall influence at time $\beta T$ is only possible if $\beta \in (\alpha_n,1]$. 
	In particular, for $\alpha_n = 1$, it cannot happen, in contradiction to our assumption. 
	
	We again use Assumption~\ref{assumpt:2.1.1} for $x^f_{\alpha T}(T)$ and notice that for $\beta \in (\alpha_n+\eps,1]$, it holds 
	\begin{equation}
	\xi_\eta T - f((1-\beta)T) \leq \left ( 1-2\sqrt{\tfrac{\alpha+\eta}{\beta}} \right ) \beta T - \tfrac{1}{2} K(\eps)T 
	\end{equation}
	for $\eta$ small enough. Suppose $\eps \leq \alpha_n$ and let $\beta = \alpha_n+\delta$ for $\delta \in (0,\eps]$. 
	Then, we obtain
	\begin{equation}
	\xi_\eta T - f((1-\beta)T) \leq \left (1-2\sqrt{\tfrac{\alpha+\eta}{\beta}} \right) \beta T + \sqrt{\tfrac{\alpha_n+\delta}{\alpha}} \eta T - \left(\sqrt{\tfrac{\alpha}{\alpha_n}}-\sigma\right)^{-1} \eta T- \tfrac{\sqrt{\alpha}}{24 \alpha_n^{3/2}}\delta^2 T,
	\end{equation}
	up to $\mathcal{O}(T^{1/3})$. For some constant $c > 0$, we find 
	\begin{equation}
	\begin{aligned}
	&\left(\left(\sqrt{\tfrac{\alpha}{\alpha_n}}-\sigma\right)^{-1} - \sqrt{\tfrac{\alpha_n+\delta}{\alpha}}  \right)\eta T +  \tfrac{\sqrt{\alpha}}{24 \alpha_n^{3/2}}\delta^2 T\geq (2 c  - \tfrac{1}{2\sqrt{\alpha\alpha_n}} \delta) \eta  T+\tfrac{\sqrt{\alpha}}{24 \alpha_n^{3/2}}\delta^2 T.
	\end{aligned} 
	\end{equation}
	If we suppose $\eta \leq \tfrac{2 \alpha^{3/2}}{3 \sqrt{\alpha_n}}c$, then the right hand side is $\geq c \eta T +  \tfrac{\sqrt{\alpha}}{24 \alpha_n^{3/2}}(\delta-6 \tfrac{\alpha_n}{\alpha}\eta)^2 T$. 
	We conclude that $x^f_{(\alpha+\eta)T}(T)$ cannot be influenced by the wall, which is contradictory. Thus, we obtain $\rho^f_l(\xi) = \sqrt{\alpha \alpha_n^{-1}}$. 
\end{proof}

\section{Weak convergence and further auxiliary results} \label{sect:auxiliary results}

In this section, we prove auxiliary results for the proofs of Theorem~\ref{th:2.3}, Proposition~\ref{pro:limit_thm_2.2}, and Corollary~\ref{cor:prop_2.2}. The first one is a weak convergence result for the process of rescaled particle positions at a fixed time in a TASEP with step initial condition.

For $(x(t),t \geq 0)$ denoting a TASEP with step initial condition, we define 
\begin{equation}
X_T(\tau) = \frac{x_{\alpha T + \hat{c}_2 \tau T^{2/3}}(T) - ((1-2\sqrt{\alpha})T - \tfrac{1}{\sqrt{\alpha}} \hat{c}_2 \tau T^{2/3} )}{-c_1 T^{1/3}},
\end{equation}
where $c_1 = (1-\sqrt{\alpha})^{2/3}\alpha^{-1/6} $ and $\hat{c}_2 = 2 \alpha^{2/3} (1-\sqrt{\alpha})^{1/3}$. 

\begin{lem} \label{lem:weak_conv} 
	It holds 
	\begin{equation} (X_T(\tau)) \Rightarrow (\mathcal{A}_2(\tau)-\tau^2) 
	\end{equation} 
	weakly with respect to the uniform topology on compact sets. 
\end{lem}

Having Lemma~\ref{lem:weak_conv}, one can readily prove a functional slow decorrelation result for $(X_T(\tau))$ by the same means as in \cite{FG24}. Together, weak convergence and functional slow decorrelation constitute the \qq{counterparts} of Proposition~2.9 of \cite{BBF21} and Proposition~5.1 of \cite{FG24}, which provide the corresponding statements for the tagged particle process. However, since it is not relevant to this study, we exclude the slow decorrelation here. 

For the weak convergence results, a key ingredient is the comparison to a stationary TASEP. For both $(X_T(\tau))$ and the tagged particle process, this comparison is established in Section~3.2 of \cite{FG24}. 

\begin{proof}[Proof of Lemma~\ref{lem:weak_conv}] 
	The claimed convergence is well-known for the finite-dimensional distributions, see \cite{BF07} for a more general framework of particle positions, and has been demonstrated in the weak sense for related models in \cite{FO17,Jo03b}. We transfer their approach to our setting and show tightness of $(X_T(\tau))$ using the criteria by Billingsley \cite{Bil68}. Together with the convergence of the finite-dimensional distributions, it implies weak convergence with respect to the uniform topology on compact sets, see Theorem~15.1 of \cite{Bil68}.
	
	Without loss of generality, we consider the interval $[0,\varkappa]$. Clearly, $X_T(0)$ is tight. We define the modulus of continuity 
	\begin{equation}
	w_T(\delta) \coloneqq \sup_{\tau_2 < \tau_1 \in [0,\varkappa], |\tau_1-\tau_2| \leq \delta} |X_T(\tau_1)-X_T(\tau_2)|
	\end{equation}
	and let $\eps, \eta > 0$. We claim that we find some $\delta > 0$ such that for $T$ large enough, 
	\begin{equation}
	\mathbb{P}(w_T(\delta) > \eps) < \eta.
	\end{equation}
	This can be proven by a comparison to particle distances in a stationary TASEP, see Lemma~3.14 and Proposition~3.16 of \cite{FG24}. More precisely, we bound the particle distances in the TASEP with step initial condition from above and from below by those in two stationary TASEPs with suitable densities. Then, we split the domain up into small subintervals. In the stationary case, the particle distances can be written as sums of geometric random variables. The resulting expression can be bounded by Doob's submartingale inequality, which provides an exponential decay of the probability as $\delta \downarrow 0$. We refer to Lemma~\ref{lem:appendix_modulus_continuity} for the detailed computations.
\end{proof} 

Using Lemma~\ref{lem:weak_conv}, we can prove Proposition~\ref{pro:limit_thm_2.2} (and Corollary~\ref{cor:prop_2.2}) by narrowing down the region of labels in \eqref{eq:lemma_2.1_sep98} that contribute non-trivially to the limiting distribution. This is the content of Lemma~\ref{lem:6.3} below. A similar argument is needed in the proof of Theorem~\ref{th:2.3}, see Corollary~\ref{cor:6.3}. 

Below, we shift the labels of the initial conditions fulfilling Assumption~\ref{assumpt:2.9} or the setting of Corollary~\ref{cor:prop_2.2} by $1$ in order to match the notation $x_{1+j}^f$ from \eqref{eq:lemma_2.1_sep98}. This does not affect the asymptotics. 

\begin{lem} \label{lem:6.3} 
	Suppose Assumption~\ref{assumpt:2.9} is fulfilled and define $ J_\varkappa \coloneqq \{(\alpha - d^{-2})T - \varkappa T^{2/3}, \dots, (\alpha-d^{-2})T + \varkappa T^{2/3}\}$.
	Then, it holds 
	\begin{equation} \label{eq:lem_6.3_0} 
	\lim_{\varkappa \to \infty}  \lim_{T \to \infty} \mathbb{P}(x_{\alpha T-j}(T) > (1-d^{-1}-d\alpha)T-ST^{1/3}-x_{1+j}^f \ \forall j \in [\alpha T-1] \setminus J_\varkappa) = 1.
	\end{equation}
\end{lem} 

The convergence \eqref{eq:lem_6.3_0} is the \qq{counterpart} of Lemma~4.11 -- Lemma~4.14 of \cite{BBF21}, which treat the tagged particle process instead of the process with varying labels. In the regions of moderate deviation from $(\alpha-d^{-2})T$, we again require a comparison to a stationary TASEP. 

\begin{proof}
	By Assumption~\ref{assumpt:2.9}, we have 
	\begin{equation}
	x_{1+j}^f \geq - d j - \tfrac{d^3}{2(1+\sqrt{\alpha}d)^2} (j-(\alpha-d^{-2})T)^2 T^{-1}.
	\end{equation}
	We set $I_\varkappa \coloneqq \{ -(\alpha-d^{-2})T, \dots, - \varkappa T^{2/3}-1 \} \cup \{ \varkappa T^{2/3}+1, \dots, d^{-2}T -1 \}$ and write
	\begin{equation}
	\begin{aligned}
	& \mathbb{P}(x_{\alpha T-j}(T) > (1-d^{-1}-d\alpha)T-ST^{1/3}-x_{1+j}^f \ \forall j \in [\alpha T-1] \setminus J_\varkappa) \\
	&= \mathbb{P}(x_{d^{-2}T-k}(T) > (1-d^{-1}-d\alpha)T-ST^{1/3}-x_{1+(\alpha-d^{-2})T+k}^f \ \forall k \in I_\varkappa).
	\end{aligned} 
	\end{equation}
	The right hand side in the probability is smaller or equal to 
	$(1-2d^{-1})T- S T^{1/3}+dk + \tfrac{d^3}{2(1+\sqrt{\alpha}d)^2}k^2 T^{-1}$. Further, for $d \geq \tfrac{1}{\sqrt{\alpha}}$ and $-(\alpha-d^{-2}) T \leq k \leq d^{-2}T-1$, it holds 
	\begin{equation}
	(1-2\sqrt{d^{-2}-kT^{-1}})T \geq (1-2d^{-1})T+dk+\tfrac{d^3}{(1+\sqrt{\alpha}d)^2}k^2 T^{-1}. 
	\end{equation}
	Thus, we obtain
	\begin{equation}
	\begin{aligned} 
	&(1-d^{-1}-d\alpha)T-ST^{1/3}-x_{1+(\alpha-d^{-2})T+k}^f \\
	& \leq (1-2\sqrt{d^{-2}-kT^{-1}})T - S T^{1/3} - \tfrac{d^3}{2(1+\sqrt{\alpha}d)^2}k^2 T^{-1}.
	\end{aligned} 
	\end{equation}
	We deduce, for $\iota \in (0,\tfrac{1}{3})$ and $\eta > 0$ small such that $1-2\sqrt{\eta} > 1-d^{-1} + \tfrac{1}{2d(1+\sqrt{\alpha}d)^2}$: 
	\begin{equation}
	\begin{aligned}
	&\mathbb{P}(\exists k \in I_\varkappa, |k| \geq \varkappa T^{2/3+\iota}: \ x_{d^{-2}T-k}(T) \leq (1-d^{-1}-d\alpha)T-ST^{1/3}-x^f_{1+(\alpha-d^{-2})T+k}) 
	\\ 
	&\leq \mathbb{P} ( x_{\eta T}(T) \leq  (1-2d^{-1})T- S T^{1/3}+d^{-1}T + \tfrac{d^3}{2(1+\sqrt{\alpha}d)^2}d^{-4} T ) \\
	&\hphantom{\leq } + \sum_{k \in I_\varkappa^{\eta,\iota}} 
	\mathbb{P} \bigl( x_{d^{-2}T-k}(T) \leq (1-2\sqrt{d^{-2}-kT^{-1}})T-ST^{1/3} - \tfrac{d^3}{2(1+\sqrt{\alpha}d)^2} \varkappa^2 T^{1/3+2\iota} \bigr) \\
	&\leq C e^{-c T^{2/3}} + C T e^{-c \varkappa^2 T^{2\iota}},
	\end{aligned} 
	\end{equation}
	where $I_\varkappa^{\eta,\iota} = \{-(\alpha-d^{-2})T, \dots, -\varkappa T^{2/3+\iota} \} \cup \{ \varkappa T^{2/3+\iota}, \dots, (d^{-2}-\eta)T\}$.  
	In the last step, we applied one-point estimates, see Lemma~\ref{Lemma_A.2}. The constants $C,c>0$ are uniform because $d^{-2}-k T^{-1} \in [\eta, \alpha] \subseteq (0,1)$ for $k \in I_\varkappa^{\eta,\iota}$. It remains to bound 
	\begin{equation}
	\begin{aligned}
	& \mathbb{P}(\exists k \in I_\varkappa, |k| \leq \varkappa T^{2/3+\iota}: \ x_{d^{-2}T-k}(T) \leq (1-d^{-1}-d\alpha)T-ST^{1/3}-x^f_{1+(\alpha-d^{-2})T+k}) \\
	&\leq \mathbb{P}(\exists k \in I_\varkappa, |k| \leq \varkappa T^{2/3+\iota}: \ x_{d^{-2}T-k}(T) \leq (1-2\sqrt{d^{-2}-kT^{-1}})T -  \tfrac{d^3k^2}{4(1+\sqrt{\alpha}d)^2} T^{-1})
	\end{aligned} 
	\end{equation}
	by a term that converges to zero in the double limit.
	Indeed, for large $\varkappa$, the second probability can be bounded by $Ce^{-c\varkappa}$ by using the comparison to a stationary TASEP, as done in the proof of Theorem~4.3 of \cite{FG24}. We use the same comparison results like for Lemma~\ref{lem:appendix_modulus_continuity}. The idea is to divide $\{\varkappa T^{2/3}+1,\dots, \varkappa T^{2/3+\iota}\}$ and $\{- \varkappa T^{2/3+\iota}, \dots,- \varkappa T^{2/3}-1\}$ into subsets of size $\varkappa T^{2/3}$, to control the particles with labels at their boundaries by the one-point estimates and to bound the particle distances for labels inside the subsets by those in a stationary TASEP with suitable density. Considering the stationary process, we again rewrite the particle distances as sums of geometric random variables and apply Doob's submartingale inequality. For the detailed computations, we refer to Lemma~\ref{lem:A.1}. 
	
	Taking first $T \to \infty$ and then $\varkappa \to \infty$, we conclude \eqref{eq:lem_6.3_0}.  
\end{proof}

\begin{cor} \label{cor:6.3} 
	For $d = \tfrac{1}{\sqrt{\alpha}}$, Assumption~\ref{assumpt:2.7} is consistent with  Assumption~\ref{assumpt:2.9}, and \eqref{eq:lem_6.3_0} remains valid for $[ \varkappa T^{2/3}]$ instead of $J_\varkappa$. This means that 
	\begin{equation}
	\lim_{\varkappa \to \infty} \lim_{T \to \infty} \mathbb{P}(x_{\alpha T-j}(T) > (1-2\sqrt{\alpha})T-x_{1+j}^f-ST^{1/3} \ \forall j \in [\alpha T-1] \setminus [\varkappa T^{2/3}]) = 1,
	\end{equation}
	as used in the proof of Theorem~\ref{th:2.3}. 
\end{cor} 

We conclude the section by proving Proposition~\ref{pro:limit_thm_2.2} and Corollary~\ref{cor:prop_2.2}. 

\begin{proof}[Proof of Proposition~\ref{pro:limit_thm_2.2}] 
	We recall that \eqref{eq:lemma_2.1_sep98} and Lemma~\ref{lem:colour_position_symmetry_step} imply
	\begin{equation} \label{eq:pf_prop_limit_thm_2.2_0} 
	\tilde{x}_{\alpha T}(T) \overset{(d)}{=} \min_{j \in [\alpha T-1]}\{x_{\alpha T-j}(T)+x_{1+j}^f \}.
	\end{equation}
	Thus, we want to compute
	\begin{equation} \label{eq:pf_prop_limit_thm_2.2_0.3}
	\lim_{T \to \infty} \mathbb{P}(x_{\alpha T-j}(T) > g_\alpha T-ST^{1/3}-x_{1+j}^f \ \forall j \in [\alpha T-1])
	\end{equation}
	for a suitable choice of $g_\alpha$. 
	
	First, let Assumption~\ref{assumpt:2.9} be fulfilled. For $g_\alpha = 1-d^{-1}-d\alpha$, Lemma~\ref{lem:6.3} tells us that in order to determine \eqref{eq:pf_prop_limit_thm_2.2_0.3}, we only need to consider the region $j \in J_\varkappa$ and to take $\varkappa \to \infty$ afterwards. We write
	\begin{equation}
	\begin{aligned} 
	& \mathbb{P}(x_{\alpha T-j}(T) > (1-d^{-1}-d\alpha)T-ST^{1/3}-x_{1+j}^f \ \forall j \in J_\varkappa ) \\
	&= \mathbb{P} \left( \sup_{\tau \in [-\hat{c}_2^{-1} \varkappa, \hat{c}_2^{-1} \varkappa]} \{ X_T(\tau) - y_T(-\tau) \} < c_1^{-1} S \right),
	\end{aligned} 
	\end{equation}
	where 
	\begin{equation} 
	X_T(\tau) = \frac{x_{d^{-2} T+\hat{c}_2 \tau T^{2/3}}(T)-((1-2d^{-1})T- d \hat{c}_2 \tau T^{2/3})}{-c_1 T^{1/3}}.
	\end{equation}
	The following arguments are similar to those in the proof of Theorem~2.5 of \cite{FG24}. By assumption, we have $\tau^2+y(-\tau) \geq (1-2(1+\sqrt{\alpha}d)^{-2})\tau^2 \geq \tfrac{1}{2} \tau^2$. Thus, Lemma~6.6 of \cite{FG24} and Lemma~\ref{lem:weak_conv} imply, together with Proposition~4.4 of \cite{CH11}, see also Proposition~2.13(b) of \cite{CLW16}, that 
	\begin{equation}
	\begin{aligned}
	& \lim_{T \to \infty} \mathbb{P}(x_{\alpha T-j}(T) > (1-d^{-1}-d\alpha)T-ST^{1/3}-x_{1+j}^f \ \forall j \in [\alpha T-1]) \\
	&= \lim_{\varkappa \to \infty} \lim_{T \to \infty} \mathbb{P} \left( \sup_{\tau \in [-\hat{c}_2^{-1} \varkappa, \hat{c}_2^{-1} \varkappa]} \{ X_T(\tau) - y_T(-\tau) \} < c_1^{-1} S \right) \\ 
	&= \mathbb{P} \left( \sup_{\tau \in \mathbb{R}} \{ \mathcal{A}_2(\tau) - \tau^2 - y(-\tau) \} < c_1^{-1} S \right). 
	\end{aligned} 
	\end{equation}
	We obtain \eqref{eq:prop_limit_thm_2.2_0}. Note that the choice $g_\alpha  = 1-d^{-1}-d\alpha$ was correct because it produces a non-trivial limiting distribution. 
	
	Next, suppose Assumption~\ref{assumpt:2.7} is fulfilled and set $g_\alpha = 1-2\sqrt{\alpha}$. Then, \eqref{eq:prop_limit_thm_2.2_1} and \eqref{eq:prop_limit_thm_2.2_2} follow by the same arguments as those in the proof of Theorem~\ref{th:2.3}, only that now, there is no contribution of a wall influence. This means that we use the weak convergence in Lemma~\ref{lem:weak_conv} and do not consider space-like paths as before. 
	Omitting the contribution of the tagged particle process in this way, we obtain \eqref{eq:prop_limit_thm_2.2_1} by the arguments in the second part of the proof of Theorem~\ref{th:2.3} for $d=\tfrac{1}{\sqrt{\alpha}}$. 
	For \eqref{eq:prop_limit_thm_2.2_2}, we argue that the limit distribution equals 
	\begin{equation}
	\lim_{T \to \infty} \mathbb{P}(x_{\alpha T-j}(T) > (1-2\sqrt{\alpha})T-S T^{1/3}-x_{1+j}^f \ \forall j \in [T^{1/3+\eps}]),
	\end{equation}
	for any fixed $\eps > 0$. We bound the term from above by the limit distribution of the $\alpha T$-th particle, which equals $F_{\textup{GUE}}(c_1^{-1}S)$. For the lower bound, we use 
	\begin{equation}
	\lim_{\kappa \to 0} \mathbb{P} \left( \sup_{\tau \in [-\kappa,0]} \{ \mathcal{A}_2(\tfrac{\tau}{\alpha})-\tfrac{1}{2}(\tfrac{\tau}{\alpha})^2 \} < c_1^{-1}S \right) = \mathbb{P}(\mathcal{A}_2(0) < c_1^{-1}S)=F_{\textup{GUE}}(c_1^{-1}S)
	\end{equation}
	like in the proof of Theorem~\ref{th:2.3}. 
\end{proof}

\begin{proof}[Proof of Corollary~\ref{cor:prop_2.2}]
	By \eqref{eq:lemma_2.1_sep98} and Lemma~\ref{lem:colour_position_symmetry_step}, we have 
	\begin{equation}
	\tilde{x}_{d^{-2}T}(T) \overset{(d)}{=} \min_{j \leq d^{-2}T-1} \{ x_{d^{-2}T-j}(T)+x_{1+j}^f \}.
	\end{equation}
	Analogously to the proof of Proposition~\ref{pro:limit_thm_2.2}, we find 
	\begin{equation}
	\begin{aligned}
	&\lim_{\varkappa \to \infty} \lim_{T \to \infty} \mathbb{P}(x_{d^{-2}T-j}(T) > (1-2d^{-1})T-ST^{1/3}-x_{1+j}^f \ \forall |j| \leq \varkappa T^{2/3} ) \\
	&= \mathbb{P} \left( \sup_{\tau \in \mathbb{R}} \{ \mathcal{A}_2(\tau) - \tau^2 - y(-\tau) \} < c_1^{-1} S \right).
	\end{aligned} 
	\end{equation}
	In order to obtain Corollary~\ref{cor:prop_2.2}, we show 
	\begin{equation}
	\lim_{\varkappa \to \infty} \lim_{T \to \infty} \mathbb{P}(x_{d^{-2}T-j}(T) > (1-2d^{-1})T-ST^{1/3}-x_{1+j}^f \ \forall j \leq d^{-2}T-1,|j| \geq \varkappa T^{2/3} )=1.
	\end{equation}
	Because for $-(1-d^{-2})T \leq j \leq d^{-2}T-1$, it holds 
	$-x_{1+j}^f \leq d j + \tfrac{1}{2} \tfrac{d^3}{(1+d)^2}j^2T^{-1}$ and 
	$(1-2\sqrt{d^{-2}-jT^{-1}})T \geq (1-2d^{-1})T+dj+\tfrac{d^3}{(1+d)^2}j^2T^{-1}$, we obtain 
	\begin{equation}
	\lim_{\varkappa \to \infty} \lim_{T \to \infty} \mathbb{P}(\exists j \in I_\varkappa:  x_{d^{-2}T-j}(T) \leq (1-2d^{-1})T-ST^{1/3}-x_{1+j}^f ) = 0
	\end{equation}
	by similar means as in the proof of Lemma~\ref{lem:6.3}, where $I_\varkappa = \{ -(1-d^{-2}-\eta)T, \dots, d^{-2} T-1 \} \setminus \{-\varkappa T^{2/3}, \dots, \varkappa T^{2/3} \}$ and $\eta \in (0,1-d^{-2})$. Further, for $\tilde{I}_\eta = \{ -(1-d^{-2})T , \dots,  -(1-d^{-2}-\eta)T\}$, we can bound
	\begin{equation}
	\begin{aligned} 
	& \mathbb{P}\bigl(\exists j \in \tilde{I}_\eta:  x_{d^{-2}T-j}(T) \leq (1-2d^{-1})T-ST^{1/3}-x_{1+j}^f \bigr) \\
	& \leq \mathbb{P}(x_T(T) \leq (1-2\sqrt{1-\eta})T-ST^{1/3}-\tfrac{d^3}{2 (1+d)^2}(1-d^{-2}-\eta)^2T)
	\end{aligned} 
	\end{equation}
	and choose $\eta $ small enough such that the right hand side in the second probability is smaller than $ -T$. Then, the probability equals zero. For $j \leq -(1-d^{-2})T$, we have 
	$x_{d^{-2}T-j}(T) \geq -d^{-2}T+j$ and 
	\begin{equation} \begin{aligned} 
	(1-2d^{-1})T-x_{1+j}^f
	& \leq  (1-2d^{-1})T-x_{1-(1-d^{-2})T}^f+j+(1-d^{-2})T  \\
	& \leq (2-d^{-2}-d^{-1}-d + \tfrac{1}{2} \tfrac{d^3}{(1+d)^2}(1-d^{-2})^2)T +j \\
	& < -d^{-2}T+j
	\end{aligned} \end{equation}
	since $d > 1$. This implies 
	\begin{equation}
	\mathbb{P}(\exists j \leq -(1-d^{-2})T :  x_{d^{-2}T-j}(T) \leq (1-2d^{-1})T-ST^{1/3}-x_{1+j}^f )=0
	\end{equation}
	and concludes the proof of Corollary~\ref{cor:prop_2.2}.
\end{proof}

\appendix 
\section{One-point estimates and some computations} \label{appendix}

Before establishing the comparisons to a stationary TASEP that are used in the proofs of Lemma~\ref{lem:weak_conv} and Lemma~\ref{lem:6.3}, we recall the exponential bounds on the one-point distributions of a TASEP with step initial condition. 

\begin{lem}[Lemma A.2 of~\cite{BBF21}] \label{Lemma_A.2}
	Let $(x(t),t\geq 0)$ be a TASEP with step initial condition and let $\alpha \in (0,1)$. Then, it holds
	\begin{equation}
	\lim_{T \to \infty} \mathbb{P}( x_{\alpha T}(T) \geq (1-2 \sqrt{\alpha})T - s c_1(\alpha) T^{1/3}) = F_{\textup{GUE}}(s), \label{one_point_limit_GUE}
	\end{equation} where $c_1(\alpha) = \tfrac{(1-\sqrt{\alpha})^{2/3}}{\alpha^{1/6}}$.
	
	In addition, we have the following estimates on the lower and the upper tail:
	uniformly for all large times $T$ and for $\alpha$ in a closed subset of $(0,1)$, there exist constants $C,c > 0$ such that
	\begin{equation}
	\mathbb{P}( x_{\alpha T}(T) \leq (1-2\sqrt{\alpha})T - s c_1(\alpha) T^{1/3}) \leq C e^{-c s} \ \ \textrm{for} \ s > 0 \label{eq:firsttail}
	\end{equation}
	and
	\begin{equation}
	\mathbb{P}(x_{\alpha T}(T) \geq (1-2 \sqrt{\alpha})T + s c_1(\alpha) T^{1/3}) \leq C e^{-c s^{3/2}} \ \ \textrm{for} \ 0 < s = o(T^{2/3}). \label{eq:secondtail}
	\end{equation}
\end{lem}

\begin{lem} \label{lem:A.1} As claimed in the proof of Lemma~\ref{lem:6.3}, 
	\begin{equation} \begin{aligned} 
	\mathbb{P}\left(\exists k \in I_\varkappa, |k| \leq \varkappa T^{2/3+\iota}: \ x_{d^{-2}T-k}(T) \leq \left(1-2\sqrt{d^{-2}-kT^{-1}}\right)T -\tfrac{d^3 k^2}{4(1+\sqrt{\alpha}d)^2} T^{-1}\right)
	\end{aligned} \end{equation} 
	can be bounded by $C e^{-c\varkappa}$ for all $T$ and $\varkappa$ large enough. 
\end{lem}
\begin{proof}
	We first bound 
	\begin{equation} \label{eq:pf_lem_A.1_0}
	\mathbb{P} \left( \exists \tau \in [c_l \varkappa, c_l \varkappa T^\iota]: \ x_{d^{-2}T-\tau T^{2/3}}(T) \leq \left(1-2\sqrt{d^{-2}-\tau T^{-1/3}}\right)T-\tfrac{d^3 \tau^2}{4(1+\sqrt{\alpha}d)^2}  T^{1/3}\right) 
	\end{equation}
	by the sum over 
	\begin{equation} \label{eq:pf_lem_A.1_0.5}
	\begin{aligned}
	\mathbb{P} \left ( \min_{\tau \in [\ell \varkappa, (\ell+1)\varkappa]} \left \{ x_{d^{-2}T-\tau T^{2/3}}(T) - \left(1-2\sqrt{d^{-2}-\tau T^{-1/3}}\right)T \right\} \leq -\tfrac{d^3  \ell^2 \varkappa^2}{4(1+\sqrt{\alpha}d)^2} T^{1/3}\right )
	\end{aligned} 
	\end{equation}
	for $\ell \in \{c_l, \dots, c_l T^\iota -1\}$ and a constant $c_l >0$ to be chosen later. 
	The summands \eqref{eq:pf_lem_A.1_0.5} are bounded from above by 
	\begin{equation}
	\begin{aligned}
	&\mathbb{P}\Bigl(x_{d^{-2}T-(\ell+1)\varkappa T^{2/3}}(T) \leq \Bigl(1-2\sqrt{d^{-2}-(\ell+1)\varkappa T^{-1/3}}\Bigr)T - \tfrac{d^3 \ell^2 \varkappa^2 }{8(1+\sqrt{\alpha}d)^2} T^{1/3}\Bigr) \\
	&+ \mathbb{P} \Bigl ( \min_{\tau \in [\ell \varkappa, (\ell+1)\varkappa]} \Bigl\{ x_{d^{-2}T-\tau T^{2/3}}(T) - x_{d^{-2}T-(\ell+1)\varkappa T^{2/3}}(T) \\ 
	&\hphantom{+ \mathbb{P} \Bigl ( } + 2 \Bigl(\sqrt{d^{-2}-\tau T^{-1/3} \vphantom{(}} - \sqrt{d^{-2}-(\ell+1)\varkappa T^{-1/3}} \Bigr)T\Bigr \} \leq - \tfrac{d^3 \ell^2 \varkappa^2}{8(1+\sqrt{\alpha}d)^2}  T^{1/3} \Bigr ). 
	\end{aligned} 
	\end{equation}
	By one-point estimates, the first probability is $\leq C e^{-c \ell^2 \varkappa^2}$ with uniform constants $C,c>0$ for all $T$ large enough. 
	We bound the second probability by a comparison to a stationary TASEP. In doing so, we set $N = d^{-2}T-\ell \varkappa T^{2/3}$, $P = d^{-2}T-(\ell+1)\varkappa T^{2/3}$, $N_\tau = d^{-2} T - \tau T^{2/3}$, $\rho_0 = \sqrt{d^{-2}-\ell \varkappa T^{-1/3}}$, $\rho_- = \rho_0 - \kappa T^{-1/3}$ and $P_- = \rho_-^2 T + \tfrac{3}{2} \kappa \rho_- T^{2/3}$. Then, it holds $P_- = d^{-2} T - \ell \varkappa T^{2/3} - \tfrac{1}{2d} \kappa T^{2/3} + \mathcal{O}(\kappa \ell \varkappa T^{1/3}, \kappa^2 T^{1/3})$. We choose $\kappa =\ell \varkappa$ and $c_l \geq 2d $ such that this choice leads to $P_- < P$. Proposition~3.16 of \cite{FG24} yields 
	\begin{equation}
	\mathbb{P} (\forall \tau \in [\ell \varkappa, (\ell+1)\varkappa] \exists t \in [0,T]: N_\tau(T \downarrow t) = P(T \downarrow t)) \geq 1 - C e^{-c \kappa},
	\end{equation}
	where the backwards indices $N_\tau(T \downarrow t)$ are constructed with respect to $(x(t),t\geq 0)$, the backwards indices $P(T \downarrow t)$ are constructed with respect to a stationary TASEP $(x^{\rho}(t),t \geq 0)$ with $\rho = \rho_-$ and the processes are coupled by clock coupling, see \cite{FG24}. 
	By Lemma~3.14 of \cite{FG24}, we have 
	\begin{equation}
	\forall \tau \in [\ell\varkappa, (\ell+1)\varkappa]: x_{N_\tau}(T) - x_P(T) \geq x^\rho_{N_\tau}(T) - x^\rho_P(T)
	\end{equation}
	with probability $\geq 1-Ce^{-c \kappa}$. Notice that $\sum_{\ell=c_l}^{c_l T^\iota-1} C e^{-c\kappa} \leq C e^{-c \varkappa}$. 
	We write $x^{\rho}_{N_\tau}(T) - x^\rho_{P}(T) = - \sum_{j=1}^{((\ell+1)\varkappa-\tau)T^{2/3}} (1+Z_j)$, where $Z_j$ are independent random variables with $\mathbb{P}(Z_j = i) = \rho(1-\rho)^i, i \geq 0$. Further, 
	$ 2 (\sqrt{d^{-2}-\tau T^{-1/3}} - \sqrt{d^{-2}-(\ell+1)\varkappa T^{-1/3}} )T \geq d((\ell+1)\varkappa-\tau)T^{2/3}$ implies that it suffices to bound 
	\begin{equation}
	\mathbb{P} \Bigl ( \min_{\tau \in [\ell\varkappa,(\ell+1)\varkappa]} \Bigl \{ - \sum\nolimits_{j=1}^{((\ell+1)\varkappa-\tau)T^{2/3}} (1+Z_j-d) \Bigr \} \leq  - \tfrac{d^3}{8(1+\sqrt{\alpha}d)^2} \ell^2 \varkappa^2 T^{1/3} \Bigr ). 
	\end{equation}
	In addition, it holds $\mathbb{E}[Z_j] = \rho^{-1}-1 = d-1+(\tfrac{1}{2}d^3+d^2)\ell \varkappa T^{-1/3} + \mathcal{O}(\ell^2 \varkappa^2 T^{-2/3})$, so we can consider 
	\begin{equation} \begin{aligned} 
	& \min_{\tau \in [\ell\varkappa,(\ell+1)\varkappa]} \Bigl \{ - \sum\nolimits_{j=1}^{((\ell+1)\varkappa-\tau)T^{2/3}} (Z_j - \mathbb{E}[Z_j])            - (\tfrac{1}{2}d^3+d^2)\ell\varkappa((\ell+1)\varkappa-\tau)T^{1/3}       \Bigr  \}  \\
	&\geq \min_{\tau \in [\ell\varkappa,(\ell+1)\varkappa]} \Bigl \{ - \sum\nolimits_{j=1}^{((\ell+1)\varkappa-\tau)T^{2/3}} (Z_j - \mathbb{E}[Z_j]) \Bigr \} 
	- (\tfrac{1}{2} d^3+d^2) \ell \varkappa^2 T^{1/3} .
	\end{aligned} \end{equation}
	We choose $c_l$ large (but constant) such that $\ell$ fulfils 
	$(\tfrac{1}{2} d^3+d^2) \ell - \tfrac{d^3}{8(1+\sqrt{\alpha}d)^2}\ell^2 < - \delta \ell^2$ for some fixed $\delta > 0$. Then, we bound 
	\begin{equation}
	\mathbb{P} \Bigl ( \max \sum (Z_j - \mathbb{E}[Z_j]) \geq \delta \ell^2 \varkappa^2 T^{1/3} \Bigr ) \leq \inf_{\lambda > 0} \mathbb{E}[e^{\lambda(Z_1-\mathbb{E}[Z_1])}]^{\varkappa T^{2/3}} e^{-\lambda \delta \ell^2 \varkappa^2 T^{1/3}}
	\end{equation}
	by Doob's submartingale inequality. Computing the moment generating function of $Z_j$ and the expansion of $\mathbb{E}[e^{\lambda(Z_j-\mathbb{E}[Z_j])}]$ in $\lambda = 0$ and choosing $\lambda = T^{-1/3}$, we obtain a bound $ \leq C e^{-c\ell^2 \varkappa^2}$. Summing over $\ell$, we conclude $\eqref{eq:pf_lem_A.1_0} \leq C e^{-c\varkappa}$. Lastly, we replace $\varkappa$ by $c_l^{-1} \varkappa$ to obtain the bound for $k \in [\varkappa T^{2/3}, \varkappa T^{2/3+\iota}]$. 
	
	The arguments for the bound on  
	\begin{equation} \label{eq:pf_lem_A.1_1}
	\mathbb{P} \left( \exists \tau \in [ \varkappa, \varkappa T^\iota]: \ x_{d^{-2}T+\tau T^{2/3}}(T) \leq \left(1-2\sqrt{d^{-2}+\tau T^{-1/3}}\right)T-\tfrac{d^3 \tau^2}{4(1+\sqrt{\alpha}d)^2}  T^{1/3}\right) 
	\end{equation}
	are analogous. 
\end{proof}

\begin{lem}  \label{lem:appendix_modulus_continuity}
	In the proof of Lemma~\ref{lem:weak_conv}, we find some $\delta > 0$ such that 
	\begin{equation} \mathbb{P}(w_T(\delta) > \eps) < \eta \end{equation}
	for all $T$ large enough.
\end{lem}

\begin{proof}
	We set $N_\tau = \alpha T + \hat{c}_2 \tau T^{2/3}$ for $\tau \in [0,\varkappa]$ and define 
	$\rho_- = \sqrt{\alpha + \hat{c}_2 \varkappa T^{-1/3}}-\kappa T^{-1/3}$, $\rho_+ = \sqrt{\alpha}+\kappa T^{-1/3}$, $P = \rho_-^2 T + \tfrac{3}{2} \kappa \rho_- T^{2/3}$, $M = \rho_+^2 T - \tfrac{3}{2} \kappa \rho_+ T^{2/3}$ and $\kappa = \delta^{-2/3} \varkappa$. Then, for small $\delta$, we find $P < \alpha T $ and $M > \alpha T + \hat{c}_2 \varkappa T^{2/3}$. Applying Proposition~3.16 of \cite{FG24} for $P ,N= \alpha T+\hat{c}_2 \varkappa T^{2/3}$ and $N=\alpha T, M$, and using Lemma~3.14 of \cite{FG24}, we find 
	\begin{equation}
	\forall \tau_2 < \tau_1 \in [0,\varkappa]: \ x^{\rho_+}_{N_{\tau_2}}(T)-x^{\rho_+}_{N_{\tau_1}}(T) \leq x_{N_{\tau_2}}(T)-x_{N_{\tau_1}}(T) \leq x^{\rho_-}_{N_{\tau_2}}(T)-x^{\rho_-}_{N_{\tau_1}}(T)
	\end{equation}
	with probability $\geq 1-C e^{-c \kappa}$. This implies 
	\begin{equation}
	\begin{aligned}
	& \mathbb{P}(w_T(\delta) > \eps) \\ 
	& \leq C e^{-c\kappa} + \mathbb{P} \Bigl ( \sup_{\tau_2 < \tau_1 \in [0,\varkappa], |\tau_2-\tau_1| \leq \delta} \{ x^{\rho_-}_{N_{\tau_2}}(T)-x^{\rho_-}_{N_{\tau_1}}(T)-\tfrac{1}{\sqrt{\alpha}}\hat{c}_2 (\tau_1-\tau_2) T^{2/3}\} > \eps c_1 T^{1/3} \Bigr ) \\
	& \hphantom{\leq C e^{-c\kappa}} \ + \mathbb{P} \Bigl ( \sup_{\tau_2 < \tau_1 \in [0,\varkappa], |\tau_2-\tau_1| \leq \delta} \{ x^{\rho_+}_{N_{\tau_1}}(T)-x^{\rho_+}_{N_{\tau_2}}(T)+\tfrac{1}{\sqrt{\alpha}}\hat{c}_2 (\tau_1-\tau_2) T^{2/3}\} > \eps c_1 T^{1/3} \Bigr ).
	\end{aligned}
	\end{equation}
	We split $[0,\varkappa]$ up into subintervals of length $\delta$ and, using translation invariance of the stationary TASEP, observe 
	\begin{equation}
	\begin{aligned}
	& \mathbb{P} \Bigl ( \sup_{\tau_2 < \tau_1 \in [0,\varkappa], |\tau_2-\tau_1| \leq \delta} \{ x^{\rho_-}_{N_{\tau_2}}(T)-x^{\rho_-}_{N_{\tau_1}}(T)-\tfrac{1}{\sqrt{\alpha}}\hat{c}_2 (\tau_1-\tau_2) T^{2/3}\} > \eps c_1 T^{1/3} \Bigr ) \\
	& + \mathbb{P} \Bigl ( \sup_{\tau_2 < \tau_1 \in [0,\varkappa], |\tau_2-\tau_1| \leq \delta} \{ x^{\rho_+}_{N_{\tau_1}}(T)-x^{\rho_+}_{N_{\tau_2}}(T)+\tfrac{1}{\sqrt{\alpha}}\hat{c}_2 (\tau_1-\tau_2) T^{2/3}\} > \eps c_1 T^{1/3} \Bigr ) \\
	& \leq \frac{\varkappa}{\delta} \mathbb{P} \Bigl ( \sup_{\tau \in [0,\delta]} \{ x^{\rho_-}_{N_{0}}(T)-x^{\rho_-}_{N_{\tau}}(T)-\tfrac{1}{\sqrt{\alpha}}\hat{c}_2 \tau T^{2/3}\} > \tfrac{\eps}{2} c_1 T^{1/3} \Bigr ) \\
	& \hphantom{\leq} \ + \frac{\varkappa}{\delta} \mathbb{P} \Bigl ( \sup_{\tau \in [0,\delta]} \{ x^{\rho_+}_{N_{\tau}}(T)-x^{\rho_+}_{N_{0}}(T)+\tfrac{1}{\sqrt{\alpha}}\hat{c}_2 \tau T^{2/3}\} > \tfrac{\eps}{2} c_1 T^{1/3} \Bigr ).
	\end{aligned}
	\end{equation}
	Next, we write $ x^{\rho_-}_{N_{0}}(T)-x^{\rho_-}_{N_{\tau}}(T) = \sum_{j=1}^{\hat{c}_2\tau T^{2/3}} (1+Z^-_j)$ with $Z^-_j \sim \text{Geom}(\rho_-)$ independent and 
	$x^{\rho_+}_{N_{\tau}}(T)-x^{\rho_+}_{N_{0}}(T) = - \sum_{j=1}^{\hat{c}_2\tau T^{2/3}} (1+Z^+_j)$ with $Z^+_j \sim \text{Geom}(\rho_+)$ independent.
	It holds 
	\begin{equation} \begin{aligned} 
	& \mathbb{E}[Z_j^-] = \rho_-^{-1}-1=\tfrac{1}{\sqrt{\alpha}}-1-\tfrac{1}{2\alpha^{3/2}}\hat{c}_2 \varkappa T^{-1/3}+\tfrac{1}{\alpha} \kappa T^{-1/3} + \mathcal{O}(\delta^{-4/3} \varkappa^2 T^{-2/3}), \\
	& \mathbb{E}[Z_j^+] = \rho_+^{-1}-1=\tfrac{1}{\sqrt{\alpha}}-1 - \tfrac{1}{\alpha} \kappa T^{-1/3} + \mathcal{O}(\delta^{-4/3} \varkappa^2 T^{-2/3}).
	\end{aligned} \end{equation} By this means, we bound the sum above by 
	\begin{equation}
	\begin{aligned}
	& \frac{\varkappa}{\delta} \mathbb{P} \Bigl ( \sup_{\tau \in [0,\delta]} \Bigl\{  \sum\nolimits_{j=1}^{\hat{c}_2\tau T^{2/3}} (Z_j^- - \mathbb{E}[Z_j^-])\Bigr\} > \tfrac{\eps}{2} c_1 T^{1/3}  - \hat{c}_2 \tfrac{1}{\alpha} \kappa \delta T^{1/3} \Bigr ) \\
	&  + \frac{\varkappa}{\delta} \mathbb{P} \Bigl ( \sup_{\tau \in [0,\delta]} \Bigl\{ \sum\nolimits_{j=1}^{\hat{c}_2\tau T^{2/3}} (\mathbb{E}[Z_j^+] - Z_j^+)\Bigr\} > \tfrac{\eps}{2} c_1 T^{1/3} -  \hat{c}_2 \tfrac{1}{\alpha} \kappa \delta T^{1/3} \Bigr ).
	\end{aligned}
	\end{equation}
	Since $\kappa \delta = \delta^{1/3} \varkappa$, we can choose $\delta$ small enough such that the right hand sides are larger than $c \eps T^{1/3}$ for some constant $c > 0$. By Doob's submartingale inequality, we obtain the bound 
	\begin{equation}
	\begin{aligned}
	\frac{\varkappa}{\delta} \inf_{\lambda > 0} \mathbb{E}[e^{\lambda(Z_1^- - \mathbb{E}[Z_1^-])}]^{\hat{c}_2 \delta T^{2/3}} e^{-c \lambda \eps T^{1/3}} + \frac{\varkappa}{\delta} \inf_{\lambda > 0} \mathbb{E}[e^{\lambda(\mathbb{E}[Z_1^+] - Z_1^+)}]^{\hat{c}_2 \delta T^{2/3}} e^{-c \lambda \eps T^{1/3}}.
	\end{aligned}
	\end{equation}
	By series expansion around $\lambda = 0$, the moment generating functions are of the form $\text{exp}({\tfrac{1-\rho}{2\rho^2}\lambda^2+\mathcal{O}(\lambda^3)})$ for $\rho = \rho_\pm$. Choosing $\lambda = \delta^{-1/2} T^{-1/3}$, the bound becomes
	\begin{equation}
	\frac{\varkappa}{\delta} C e^{-c \delta^{-1/2}\eps}.
	\end{equation}
	Thus, we have shown 
	\begin{equation}
	\mathbb{P}(w_T(\delta) > \eps ) < C e^{-c \delta^{-2/3} \varkappa} + \frac{\varkappa}{\delta} C e^{-c \delta^{-1/2}\eps}
	\end{equation}
	for all $T$ large enough. Choosing $\delta$ small enough, the right hand side becomes smaller than $\eta$.
	We wish to point out that our choices of $\lambda$ and $\kappa$ do not lead to an optimal bound, but to one that is enough for our purpose.  
\end{proof}

\end{document}